\documentclass{article}
\usepackage{amsmath}
\usepackage{amsfonts}
\usepackage{amsthm}
\usepackage{enumerate}

\newtheorem{thm}{Theorem}
\newtheorem{prop}[thm]{Proposition}
\newtheorem{lem}[thm]{Lemma}
\newtheorem{cor}[thm]{Corollary}

\newtheorem*{thmnn}{Theorem}

\theoremstyle{definition}
\newtheorem{defn}[thm]{Definition}
\newtheorem{rmk}[thm]{Remark}

\begin{document}

\title{Polynomial Divided Difference Operators Satisfying the Braid Relations}

\author{Shaul Zemel}

\maketitle


\section*{Introduction}

Many interesting families of polynomials are indexed by permutations, and are constructed using divided difference operators and their generalizations. The most classical one is the Schubert polynomials, as introduced in \cite{[LS1]}, and their doubled version from \cite{[Mac]}. The former show up naturally in the cohomology theory of flag varieties, see, e.g., \cite{[FP]}, yields the quantum generalization from \cite{[FGP]}, and the universal one from \cite{[F]} (the quantum and universal Schubert polynomials cannot be constructed using the usual divided difference operators, and require variants which are not considered in the current paper, but their double version of the former can be obtained in this way by acting on the second set of variables---see Appendix J of \cite{[FP]}). A recent generalization is the twisted Schubert polynomials from \cite{[Li]}, which is constructed using modified divided difference operators, in the form considered here.

Another generalized version of the Schubert polynomials consists of the Grothendieck polynomials, originally defined in \cite{[LS2]} (see also \cite{[La1]}). They can also be constructed using divided difference operators, but their operation is combined with multiplication by additional polynomials. They too have a double version, as well as a quantum version (see \cite{[LM]}) and a universal one (as in \cite{[BKTY]}). We also mention the Grothendieck $H$-polynomials from \cite{[LS3]}, which admit a similar definition. Note that like the Stanley symmetric functions (originally defined in \cite{[S]}) show up as a stable limit of the Schubert polynomials, there are also stable Grothendieck polynomials (defined in \cite{[FK1]}), and when the permutation has only one descent, so that the Schubert polynomial is a Schur polynomial (and the Stanley symmetric function is a Schur symmetric function), the stable Grothendieck polynomials admit duals, as considered in \cite{[LP]}. Generalizations of the latter are investigated in \cite{[Ye]}.

The Schubert and Grothendieck polynomials associated with a permutation $w$, with their doubled versions, are the generating functions, or partition functions with appropriate weights, of many interesting objects. To name a few, we mention on one side pipe dreams (or RC-graphs, or compatible sequences), as in \cite{[FK3]}, \cite{[BJS]}, \cite{[BB]}, and \cite{[FS]} (see also \cite{[KM]}, \cite{[Kn]}, and the appendix of \cite{[Hud]}), and on the other hand bumpless pipe dreams---see \cite{[LLS]}, \cite{[W]}, and \cite{[Hua]}. The resulting matching between the two objects is far from obvious, and was completed in the reduced (or Schubert) case in \cite{[GH]}. There is also a lattice model interpretation as given in \cite{[BS]}. The recent preprint \cite{[LOTRZ]} establishes a bumpless pipe dream model also for the quantum Grothendieck polynomials. We also mention a further generalization of both the Grothendieck polynomials and their $H$-versions, namely the biaxial double Grothendieck polynomials of \cite{[BFHTW]}, involving a more complicated modification of the divided difference operators, which are also shown there to be the partition functions of appropriate colored lattice models (yet another generalization is defined in \cite{[GZJ]}).

\medskip

Another family of polynomials, whose indices are compositions but which are constructed similarly using divided difference operators (with polynomials) via the permutation yielding the composition from the associated partition, is the family of key polynomials. They were defined in \cite{[D1]} and \cite{[D2]}, and they are constructed using the Demazure atoms from \cite{[LS4]} (defined by yet another modification of the divided difference operators), which also related them to Young tableaux. Many properties of these polynomials are presented in the thesis \cite{[P]}. Other related objects are skyline fillings (see \cite{[Mas1]} and \cite{[Mas2]}) and Gelfand--Tsetlin polytopes (as in \cite{[KST]}). In fact, the key polynomials and the Demazure atoms are connected as special cases of the non-symmetric Macdonald polynomials (see \cite{[AS]} and the papers cited therein).

The same type of generalization that produces the Grothendieck polynomials from the Schubert polynomials yields, when applied to the key polynomials and Demazure atoms, the Lascoux polynomials from \cite{[La2]} and the Lascoux atoms defined in \cite{[Mo]} respectively. The divided difference operators in this case involve quadratic polynomials. They admit, by \cite{[BSW]}, associated colored lattice models, and \cite{[Yu]} relates them to set-valued tableaux. A conjecture relating the Lascoux polynomials to the Grothendieck polynomials was posed in \cite{[RY]} and answered soon after in \cite{[SY]}.

\medskip

All these interesting families of polynomials are based, in their definitions via the divided difference operators and their modifications, on the braid relations required for making well-defined indices in terms of permutations. The closely related question, of operators satisfying the Yang--Baxter equation, more specifically in the exponential setting, was considered in \cite{[FK1]}, \cite{[FK2]}, and \cite{[FK3]}. The paper \cite{[Ki]} introduces more general divided difference operators, and considers more general families of operators and polynomials, including those mentioned above, the ones from \cite{[dFZJ]}, and others. The most general family of operators there involves five parameters, which must satisfy a quadratic equation in order for a relation similar to the braid relation to hold (see Lemma 4.14 of that reference).

\medskip

This motivates the question---which modifications of the divided difference operators by arbitrary polynomials satisfy the braid relations, in order to be possibly used for defining families of polynomials with indices that are permutations? This paper answers this question in full. Under a certain condition of non-degeneracy, the family from \cite{[Ki]} is almost the only possible choice, though there are a few smaller families of possible operators. 

\medskip

We will now explain the notions and results of this paper in more detail. For a function $f$ of variables including $x_{i}$ and $x_{i+1}$, the $i$th transposition operator $s_{i}$ takes $f$ to the function $s_{i}f$ obtained by interchanging the values of $x_{i}$ and $x_{i+1}$. Recall the \emph{divided difference operator} $\partial_{i}$, taking such a function $f$ to $\frac{f-s_{i}f}{x_{i}-x_{i+1}}$. While the fact that $f=s_{i}f$ at points where $x_{i}=x_{i+1}$ makes the expression $\partial_{i}f$ for a general smooth function $f$ well-defined at such points by taking limits and derivatives, in most applications $f$ is a polynomial, and then this property implies that $\partial_{i}f$ is also a polynomial. Moreover, as the quotient of two expressions that are anti-symmetric with respect to interchanging $x_{i}$ and $x_{i+1}$, the dependence of the function $\partial_{i}f$ on these two variables is symmetric.

Fix now some $n\geq3$, and consider functions $f$ that depend, among perhaps some extra parameters, on the variables $\{x_{i}\}_{i=1}^{n}$. On any such function $f$ we have an action of every operator $\partial_{i}$ with $1 \leq i \leq n-1$. It then well-known, and straightforward to check (and will also follow from the calculations carried out in this paper), that the operators $\{\partial_{i}\}_{i=1}^{n-1}$ satisfy the \emph{braid relations}
\begin{equation}
\partial_{i}\partial_{k}=\partial_{k}\partial_{i}\mathrm{\ if\ }|k-i|\geq2\qquad\mathrm{and}\qquad\partial_{i}\partial_{i+1}\partial_{i}=\partial_{i+1}\partial_{i}\partial_{i+1}\mathrm{\ for\ }1 \leq i \leq n-2. \label{braid}
\end{equation}
We call the first relations in Equation \eqref{braid} the \emph{quadratic braid relations}, and the second ones the \emph{cubic braid relations}. Of course, for $n=1$ the family is empty and for $n=2$ no relations are imposed, so we assume $n\geq3$ throughout.

\medskip

For various applications a modification of $\partial_{i}$ is required, usually involving combining the action on $\partial_{i}$ with multiplication by a polynomial in $x_{i}$ and $x_{i+1}$. We gather these operations, recall that $f$ and $s_{i}f$ are involved in the definition of such an operator on a function $f$, and we define in general a \emph{polynomial divided difference operator} to be an operator of the form
\begin{equation}
\pi_{i}:f\mapsto\partial_{i}\big(P(x_{i},x_{i+1})f\big)+Q(x_{i},x_{i+1})\partial_{i}f+R(x_{i},x_{i+1})f+S(x_{i},x_{i+1})s_{i}f. \label{poldifop}
\end{equation}
We allow $P$, $Q$, $R$, and $S$ in Equation \eqref{poldifop} to depend on additional parameters that are not the $x_{i}$'s (as can the function $f$ on which they operate), but for fixed $i$ the only dependence of the expression in that equation on $x_{j}$ with $j$ that equals neither $i$ nor $i+1$ is through $f$.

The most interesting case is that of the operators which are \emph{non-degenerate}, in the sense that the action of $\pi_{i}$ on $f$ is not just a multiple of $f$ or of $s_{i}f$ by some polynomial (see Remark \ref{degnondeg} for the precise definition). As a consequence of our results, we obtain a motivating characterization for the general family of operators from \cite{[Ki]}, as follows.
\begin{thmnn}
Given $n\geq3$, assume that $\{\pi_{i}\}_{i=1}^{n-1}$ is a family of non-degenerate polynomial divided difference operators as in Equation \eqref{poldifop}, with $R=S=0$ for each $i$ and $P$ and $Q$ being the same polynomial for all $i$, for which the braid relations hold. Then there are five constants $a$, $b$, $c$, $d$, and $e$ such that the first four are not all zero and satisfy $ad-bc=0$, such that for every $1 \leq i \leq n-1$ and function $f$ we have \[\pi_{i}f=(b-c-e)\partial_{i}(x_{i}f)+[ax_{i}x_{i+1}+(c+e)x_{i}+cx_{i+1}+d]\partial_{i}f.\]
\end{thmnn}
(note that our choice of parametrization differs slightly from that of \cite{[Ki]}, but is more convenient for our analysis).

Recall that the \emph{Hecke algebra} $\mathcal{H}_{\mu,\nu}$ with parameters $\mu$ and $\nu$, that is associated with $S_{n}$ (or the root system $A_{n}$) is generated by elements $\{u_{i}\}_{i=1}^{n-1}$ that satisfy the braid relations from Equation \eqref{braid}, plus the quadratic relation $u_{i}^{2}=\mu u_{i}+\nu$ for every $i$. Such an algebra is called \emph{double affine nilCoxeter algebra of type A} in \cite{[Ki]}, and generalizes the notions of the nilCoxeter and Demazure algebras from other references mentioned above. The family of operators from the last theorem form a Hecke algebra $\mathcal{H}_{b-c,e(e+c-b)}$.

\medskip

The results in this paper are more general, and for describing them in more detail we need some notions. First, the form of the operator in Equation \eqref{poldifop} is not unique, and one may wish to normalize it in a more canonical form. We describe in Proposition \ref{canform} three ways to do it, two of whose require a certain ``canonical lift'' of a polynomial that is symmetric in $x_{i}$ and $x_{i+1}$ to an appropriate pre-image under $\partial_{i}$, which we call \emph{$\partial_{i}$-positive} (see Definition \ref{dipos}). This produces two canonical polynomials associated with each such operator, namely $Q_{0}$ from Proposition \ref{canform} and $T$ from Equation \eqref{Tpol}, which for the operators from the theorem mentioned above are \[ax_{i}x_{i+1}+(c+e)x_{i}+(b-e)x_{i+1}+d\quad\mathrm{and}\quad ax_{i}x_{i+1}+bx_{i}+cx_{i+1}+d\] (the latter independently of $e$) respectively. The non-degeneracy assumption is equivalent to neither $Q_{0}$ nor $T$ vanishing.

In the most general setting, we do not require that the operators from Equation \eqref{poldifop} are the same for all $i$. However, under the non-degeneracy assumption, the braid relations imply that the polynomial $T$ must be the same for all $i$, and for $Q_{0}$ this is almost correct, namely the polynomials $Q_{0}$ associated with different indices $i$ in such a family are ``the same up to replacing one variable in a univariate divisor of the polynomial by the other variable'' (see Definition \ref{almosteq} for the precise notion). Back to our family of polynomials, $T$ is always a product of two non-zero linear (or constant) polynomials, one in each variable (by the equality $ad=bc$), and $Q_{0}$ is also such when $e=0$ or $e=b-c$, but is irreducible and contains both variables otherwise. Hence for a generic value of $e$ the conditions that all the operators are based on the same polynomials becomes redundant, but the operators with $e$-values 0 and $b-c$ are interchangeable, and they are also interchangeable with the operators taking $f$ to \[\partial_{i}[(ax_{i}^{2}+bx_{i})f]+(cx_{i+1}+d)\partial_{i}f-ax_{i}f\mathrm{\ or\ }\partial_{i}[(cx_{i+1}+d)f]+(ax_{i}^{2}+bx_{i})\partial_{i}f-ax_{i}f.\] Note that the interchangeable options show up precisely when the parameter $\nu=e(e+c-b)$ of the Hecke algebra structure vanishes.

\medskip

We also determine the families satisfying the braid relations in which some of the operators may be degenerate (but not 0, for trivial reasons). First we show that if the $T$ polynomial vanishes for one operator then it does for all of them, and the families arising in this way are precisely those in which every $\pi$ is the transposition operator $s_{i}$ multiplied by a polynomial $S_{i}$, and the $S_{i}$'s are all almost equal (see Proposition \ref{degenT}). When the $Q_{0}$-ones are allowed to vanish, the $T$-polynomials cannot vanish, and the set of indices $i$ for which $\pi_{i}$ is non-degenerate is divided into a separated collection of consecutive intervals and isolated operators. We prove the following result.
\begin{thmnn}
Assume that $n\geq4$, and that some of the $\pi_{i}$'s have vanishing $Q_{0}$ polynomials. Then all the non-degenerate $\pi_{i}$'s are $\mu\operatorname{Id}$ for some non-zero constant $\mu$, the consecutive intervals are as in the non-degenerate case with a Hecke algebra $\mathcal{H}_{\mu,0}$ (for some smaller value of $n$), and the isolated operators are of the form $f\mapsto\varphi_{i}\partial_{i}(\psi_{i}f)$ where $\partial_{i}(\varphi_{i}\psi_{i})=\mu$.
\end{thmnn}
The operators in the latter theorem also generate a Hecke algebra $\mathcal{H}_{\mu,0}$. The ones with vanishing $T$-polynomials generate a Hecke algebra only when $\pi_{i}=rs_{i}$ for a constant $r$, and then the Hecke algebra is $\mathcal{H}_{0,r^{2}}$. We remark that the hypothesis $n\geq4$ is important in the last theorem, and for $n=3$ there are additional families---see Proposition \ref{degenQ} and Remark \ref{Qdegen}  (none of which produces a Hecke algebra though). We conclude with some remarks about commutation relations between the operators in the families that we found.

\medskip

The paper is divided into 5 sections. Section \ref{CanQuad} defines the notions required for describing the canonical forms of polynomial divided difference operators, as well as discusses the quadratic braid relations. Section \ref{CubandDeg} establishes the basic equalities associated with the cubic braid relations, and determines the families with vanishing $T$-polynomials. Section \ref{NonDeg} investigates the families of non-degenerate polynomial divided difference operators satisfying the braid relations (thus establishing the first main theorem), and Section \ref{NonDeg}, and Section \ref{WithVanQ0} presents the results involving also operators for which the $Q_{0}$ polynomials might vanish (yielding the second main theorem). Finally, Section \ref{HecandCom} contains the connections with Hecke algebras and the comments about commutation relations.

\medskip

I am grateful to B. Brubaker for his encouragement to pursue the investigation of these questions, as well as to P. Alexandersson, A. Weigandt, and D. Huang for interesting discussions around related topics.

\section{Canonical Forms and the Quadratic Relations \label{CanQuad}}

We fix a number $n\geq3$ (smaller $n$'s give trivial results), and we work throughout with an unspecified field of coefficients (of arbitrary characteristic), over which all of our polynomials and operators are defined.

One of the elementary but crucial property of the divided difference operator $\partial_{i}$ is the modified Leibniz rule, as follows.
\begin{lem}
We have the equality $\partial_{i}(fg)=\partial_{i}f \cdot g+s_{i}f\cdot\partial_{i}g$ for any two functions $f$ and $g$. In particular, if $g$ is symmetric in $x_{i}$ and $x_{i+1}$ then multiplication by $g$ commutes with $\partial_{i}$, namely $\partial_{i}(fg)=\partial_{i}f \cdot g$. \label{Leibniz}
\end{lem}

\begin{proof}
The first equality is obtained by adding and subtracting $\frac{s_{i}f \cdot g}{x_{i}-x_{i+1}}$ to the expression $\frac{fg-s_{i}f \cdot s_{i}g}{x_{i}-x_{i+1}}$ defining $\partial_{i}(fg)$ (recall that $s_{i}(fg)=s_{i}f \cdot s_{i}g$) and using the definition of $\partial_{i}f$ and $\partial_{i}g$. By noting that the symmetry condition on $g$ implies that $s_{i}g=g$ and hence $\partial_{i}g=0$, the second equality follows from the first one. This proves the lemma.
\end{proof}

Lemma \ref{Leibniz} implies that given a polynomial divided difference operator, the expression from Equation \eqref{poldifop} describing it is not unique. In order to obtain a unique presentation, we make the following definition.
\begin{defn}
A monomial $x_{i}^{r}x_{i+1}^{s}$, times a monomial in other variables perhaps, is \emph{$\partial_{i}$-positive} is $r>s$. A general polynomial in $x_{i}$ and $x_{i+1}$ (among other variables) is \emph{$\partial_{i}$-positive} if it is a linear combination of $\partial_{i}$-positive monomials. \label{dipos}
\end{defn}
The reason for the terminology in Definition \ref{dipos} is that the $\partial_{i}$-positive monomials are precisely the monomials in $x_{i}$ and $x_{i+1}$ to which applying $\partial_{i}$ yields a non-zero combination of monomials with positive coefficients. Since the $\partial_{i}$-positive monomials $x_{i}^{r}x_{i+1}^{s}$ combine with the symmetric ones $x_{i}^{r}x_{i+1}^{s}+x_{i}^{s}x_{i+1}^{r}$ and $x_{i}^{r}x_{i+1}^{r}$ to produce a basis for all polynomials, it is clear that every polynomial in $x_{i}$ and $x_{i+1}$ can be uniquely presented as the sum of a polynomial that is symmetric in these variables and a $\partial_{i}$-positive one.

\begin{rmk}
Recall that if a polynomial in $x_{i}$ and $x_{i+1}$ is symmetric then its degree in $x_{i}$ and in $x_{i+1}$ separately is the same. We call this degree the \emph{variable degree} of this polynomial. Thus, for a polynomial like $x_{i}^{r}x_{i+1}^{s}+x_{i}^{s}x_{i+1}^{r}$ with $r>s$ mentioned above, its homogeneity degree is $r+s$ and its variable degree is $r$. It is clear that if such a polynomial is homogeneous of degree $m$ and has variable degree $r$ then $2r \geq m$. \label{vardeg}
\end{rmk}

Now, it is clear from the definition that the kernel of $\partial_{i}$ consists of the polynomials that are symmetric in $x_{i}$ and $x_{i+1}$ (as in the proof of Lemma \ref{Leibniz}), and since both the numerator and denominator in its definition are anti-symmetric, its image also lands in the symmetric polynomials. These properties are complemented by the following result.
\begin{lem}
The map $\partial_{i}$ is a bijection from the $\partial_{i}$-positive polynomial in $x_{i}$ and $x_{i+1}$ onto those that are symmetric in these variables. \label{symimdi}
\end{lem}

\begin{proof}
We saw that the $\partial_{i}$-positive monomials intersect the symmetric ones trivially, whence the injectivity of the map. We also know that the image consists only of symmetric polynomials, so it remains to prove surjectivity onto these polynomials.

We now note that if a polynomial is homogeneous of degree $m+1$ in $x_{i}$ and $x_{i+1}$ then the degree of homogeneity of its $\partial_{i}$-image in these variables is $m$. It thus suffice suffices to consider a polynomial $\varphi$ that is symmetric and homogeneous of degree $m$ in these variables, and find a $\partial_{i}$-positive pre-image under $\partial_{i}$ that is homogeneous of degree $m+1$. We then argue by induction on the variable degree $r$ of $\varphi$, which was seen to satisfy $2r \geq m$ (see Remark \ref{vardeg}).

The variable degree condition implies that $\varphi$ is a non-zero multiple of $x_{i}^{r}x_{i+1}^{d-r}$ plus a linear combination of monomials involving powers of $x_{i}$ that are smaller than $r$. We can thus write $\varphi$ as $a\sum_{s=d-r}^{r}x_{i}^{s}x_{i+1}^{d-s}$ for the appropriate $a\neq0$, plus a symmetric polynomial that is homogeneous of degree $m$ and has variable degree smaller than $r$ (in particular the remainder is 0 in case $2r-1 \leq m$). Now, the induction hypothesis implies that remainder is the $\partial_{i}$-image of some $\partial_{i}$-positive polynomial (which is unique and homogeneous of degree $m+1$, where for 0 we just take 0), and as $\partial_{i}(x_{i}^{r+1}x_{i+1}^{d-r})$ equals $\sum_{s=d-r}^{r}x_{i}^{s}x_{i+1}^{d-s}$, adding the $\partial_{i}$-positive monomial $ax_{i}^{r+1}x_{i+1}^{d-r}$ to that pre-image produces a pre-image of $\varphi$, as desired. This proves the lemma.
\end{proof}
The uniqueness of the $\partial_{i}$-positive pre-image of $\varphi$ under $\partial_{i}$ is also visible in the fact that the multiple of each $\partial_{i}$-positive monomial $x_{i}^{r+1}x_{i+1}^{d-r}$ in the construction of that pre-image in the proof of Lemma \ref{symimdi} was determined by $\varphi$.

A polynomial divided difference operator can have many presentations, with different polynomials $P$, $Q$, $R$, and $S$. However, there are three specific forms of the operator that are canonical, each with its own normalization.
\begin{prop}
Let $\pi_{i}$ be a polynomial divided difference operator as in Equation \eqref{poldifop}. Then there is a unique way of writing $\pi_{i}$ with $P=S=0$, another unique way of writing $\pi_{i}$ with $S=0$ and $Q$ and $R$ that are $\partial_{i}$-positive as in Definition \ref{dipos}, and a third unique way in which $S=0$ and $\partial_{i}$-positive $P$ and $R$. The polynomial $R$ in the two last presentations is the same. \label{canform}
\end{prop}
We denote the polynomials showing up in the first canonical form of $\pi_{i}$ by $Q_{0}$ and $R_{0}$, those in the second one by $P_{+}$, $Q^{+}$, and $R_{+}$, and the ones showing up in the third form by $P^{+}$, $Q_{+}$, and $R_{+}$. Indeed, we get the same $R=R_{+}$ in the latter two presentations, and we keep the notation $P_{+}$ and $Q_{+}$ to indicate $\partial_{i}$-positive polynomials, while $Q^{+}$ and $P^{+}$ indicate that the other polynomial in the presentation is $\partial_{i}$-positive.

\begin{proof}
The definition of $\partial_{i}$ implies that $s_{i}f=f-(x_{i}-x_{i+1})\partial_{i}f$. Thus we can replace $R$ by $R+S$, $Q$ by $Q-(x_{i}-x_{i+1})S$, and $S$ by 0 and get a presentation of the same polynomial divided difference operator. We may thus restrict attention to presentations in which $S=0$.

Recall now from Lemma \ref{Leibniz} that we can write $\partial_{i}(Pf)$ as $\partial_{i}P \cdot f+s_{i}P\cdot\partial_{i}f$. Thus, by replacing $Q$ by $Q_{0}:=s_{i}P+Q$, $R$ by $R_{0}:=R+\partial_{i}P$, and $P$ by 0 we obtain the first asserted form. Now, if $\pi_{i}$ is given in such a form then by substituting $f=1$ (hence $\partial_{i}f=0$) we get $\pi_{i}(1)=R_{0}$. Moreover, we have $\pi_{i}(x_{i})=Q_{0}+R_{0}x_{i}$, which determines $Q_{0}$ once $R_{0}$ is known. The canonicity of this form follows.

Next, write $R$ as $R_{+}+R_{s}$, with $R_{s}$ symmetric with respect to interchanging $x_{i}$ and $x_{i+1}$, and such that $R_{+}$ is $\partial_{i}$-positive (which can be done in a unique manner). Lemma \ref{symimdi} produces a polynomial $\widetilde{R}$ such that $\partial_{i}\widetilde{R}=R_{s}$. We write $\partial_{i}(\widetilde{R}f)$ as $R_{s}f+s_{i}\widetilde{R}\cdot\partial_{i}f$ (Lemma \ref{Leibniz}), so that replacing $P$ by $P+\widetilde{R}$, $Q$ by $Q-s_{i}\widetilde{R}$, and $R$ by $R_{+}$ yields a presentation of the same operator but in which $R=R_{+}$ is $\partial_{i}$-positive. Moreover, if $\pi_{i}$ is given in any form with $S=0$ then $\pi_{i}(1)=\partial_{i}P+R$, so that $R_{+}$ is determined as the $\partial_{i}$-positive part of $\pi_{i}(1)$.

Knowing that $S=0$ and $R=R_{+}$ is the uniquely determined $\partial_{i}$-positive option, we consider the decompositions for $P=P_{+}+P_{s}$ and $Q=Q_{+}+Q_{s}$ in the same manner, and the second assertion of Lemma \ref{Leibniz} implies that the symmetric parts $P_{s}$ and $Q_{s}$ can be taken in and out of the action of $\partial_{i}$. The second form is thus with this $P_{+}$ and with $Q^{+}:=Q+P_{s}$, and the third one involves $Q_{+}$ and $P^{+}:=P+Q_{s}$. This implies that $P_{+}+Q^{+}$ and $P^{+}+Q_{+}$ both equal $P+Q=P_{+}+Q_{+}+P_{s}+Q_{s}$.

Finally, recalling that $\pi_{i}(1)=\partial_{i}P+R$ with $R=R_{+}$, Lemma \ref{symimdi} implies that $P_{+}$ is the unique $\partial_{i}$-positive polynomial whose $\partial_{i}$-image is the symmetric part of $\pi_{i}(1)$. We also have $\pi_{i}(x_{i})=\partial_{i}(Px_{i})+Q+Rx_{i}$, from which we extract the value of $Q^{+}$ in case $R=R_{+}$ and $P=P_{+}$. This also shows that $Q_{+}$ is the $\partial_{i}$-positive part of $\pi_{i}(x_{i})-x_{i}R_{+}$, and once $P_{+}$, $Q^{+}$, and $Q_{+}$ are known, the value of $P^{+}$ is given by $P_{+}+Q^{+}-Q_{+}$ from above. Hence the second and third forms of any polynomial divided difference operator are also canonical. This proves the proposition.
\end{proof}

It follows from the constructions in Proposition \ref{canform} that $Q_{0}(x_{i},x_{i+1})$ and $R_{0}(x_{i},x_{i+1})$ can be given by $P(x_{i+1},x_{i})+Q(x_{i},x_{i+1})-(x_{i}-x_{i+1})S(x_{i},x_{i+1})$ and $R(x_{i},x_{i+1})+S(x_{i},x_{i+1})+\partial_{i}P(x_{i},x_{i+1})$ respectively, given any presentation of $\pi_{i}$ as in Equation \eqref{poldifop} (the ones with $+$ require a decomposition, and thus cannot be written directly using such formulae). These combinations are thus independent of the presentation. We now introduce a third combination that will turn out useful later. We define
\begin{equation}
T(x_{i},x_{i+1}):=P(x_{i},x_{i+1})+(x_{i}-x_{i+1})R(x_{i},x_{i+1})+Q(x_{i},x_{i+1}), \label{Tpol}
\end{equation}
and prove the following consequence of Proposition \ref{canform}, or rather of its proof.
\begin{cor}
The polynomial $T$ from Equation \eqref{Tpol} is independent of the presentation of $\pi_{i}$. \label{Tcan}
\end{cor}

\begin{proof}
One way of obtaining this property is by verifying that $T$ remains invariant under the operations in the proof of Proposition \ref{canform}. A second way one is by observing that the values of $\pi_{i}(1)$ and $\pi_{i}(x_{i})$ (to which we have to add $S$ and $Sx_{i+1}$ respectively in case $S\neq0$) imply, after expanding $\partial(Pf)$ as in Lemma \ref{Leibniz} and using similar arguments, that $T(x_{i},x_{i+1})$ is given by $\pi_{i}(x_{i})-x_{i+1}\pi_{i}(1)$ (in a similar manner we get $Q_{0}=\pi_{i}(x_{i})-x_{i}\pi_{i}(1)$ as well). A third way is via the fact that similar considerations express $\pi_{i}f$ as $\frac{T(x_{i},x_{i+1})f-Q_{0}(x_{i},x_{i+1})s_{i}f}{x_{i}-x_{i+1}}$, a formula that will also be useful below. This proves the corollary.
\end{proof}
In fact, in Proposition \ref{canform} one could consider the normalization in which $R=0$, and then $P=0$ as well. In this normalization the coefficient of $\partial_{i}f$ is $T$ from Equation \eqref{Tpol}, and $R_{0}$ shows up again, as the multiplier of $s_{i}f$.

\medskip

As Proposition \ref{canform} reduces the presentation of a general polynomial divided difference operator by the last term in Equation \eqref{poldifop}, we can simplify some calculations. First we establish the quadratic braid relations from Equation \eqref{braid} for our operators, which hold in general.
\begin{prop}
Let $i$ and $k$ be indices with $|k-i|\geq2$, and set $\pi_{i}$ and $\pi_{k}$ to be as in Proposition \ref{canform}, with polynomials that may be different for the two indices. Then we have $\pi_{i}\pi_{k}=\pi_{k}\pi_{i}$. \label{quadcom}
\end{prop}

\begin{proof}
The assumption on $i$ and $k$ implies, via the second part of Lemma \ref{Leibniz} and the fact that functions that are independent of $x_{i}$ and $x_{i+1}$ (or $x_{k}$ and $x_{k+1}$) are symmetric with respect to them, that the polynomials showing up in $\pi_{k}$ commute with $\partial_{i}$ and vice versa. For simplicity we shall use the normalization with $Q_{0}$ and $R_{0}$ (so that $P=S=0$) for both operators, and we shorthand $Q_{i,0}(x_{i},x_{i+1})$, $Q_{k,0}(x_{k},x_{k+1})$, $R_{i,0}(x_{i},x_{i+1})$ and $R_{k,0}(x_{k},x_{k+1})$ to $Q^{0}_{i}$, $Q^{0}_{k}$, $R^{0}_{i}$, and $R^{0}_{k}$ respectively. The commutation relations allow us to write \[\pi_{i}\pi_{k}f=Q^{0}_{i}Q^{0}_{k}\partial_{i}\partial_{k}f+Q^{0}_{i}R^{0}_{k}\partial_{i}f+R^{0}_{i}Q^{0}_{k}\partial_{k}f+R^{0}_{i}R^{0}_{k}f,\] and $\pi_{k}\pi_{i}f$ is given by the same expression with $i$ and $k$ interchanged. As the assumption on $i$ and $k$ implies that $s_{i}s_{k}g=s_{k}s_{i}g$ for any function $g$ that depends on $x_{i}$, $x_{i+1}$, $x_{k}$, and $x_{k+1}$, similar considerations show that $\partial_{i}\partial_{k}g=\partial_{k}\partial_{i}g$ for any such $g$, which yields the equality of our expressions for $\pi_{i}\pi_{k}f$ and $\pi_{k}\pi_{i}f$ as desired. This proves the proposition.
\end{proof}
One could, in fact, prove Proposition \ref{quadcom} without invoking Proposition \ref{canform}, but the expressions in the proof would have contained more terms but still lead to the same result by the same reasoning.

\section{The Cubic Braid Relations and Degeneracy \label{CubandDeg}}

For the cubic braid relations in Equation \eqref{braid}, only the three variables $x_{i}$, $x_{i+1}$, and $x_{i+2}$ are concerned. To ease notation we replace them by $x$, $y$, and $z$ respectively. To remove more indices, we denote the operator $s_{i}$, interchanging $x$ and $y$, by simply $s$, and write $\sigma$ for the operator $s_{i+1}$ interchanging $y$ and $z$. The cubic braid relation for permutations thus becomes $s\sigma s=\sigma s\sigma$ in this notation. We also write $\pi$ for $\pi_{i}$ and $\varpi$ for $\pi_{i+1}$, and recalling that we allow different polynomials for them in Equation \eqref{poldifop}, we denote those of $\pi=\pi_{i}$ as in that equation, and the ones showing up in $\varpi=\pi_{i+1}$ by adding tildes. We shall also write $P_{xy}$ for $P(x,y)$, $Q_{yz}$ for $Q(y,z)$, and similarly for the other polynomials, including those of the tildes. Note that $sP_{xy}$ thus equals $P_{yx}$, while $\sigma Q_{xz}=Q_{xy}$, $sR_{yz}=R_{xz}$ and so on, and pnce again for the polynomials $Q_{0}$ and $R_{0}$ we shall add the superscript 0 (like in the proof of Proposition \ref{quadcom}).

The expressions that we must consider for the cubic braid relation are the following ones.
\begin{lem}
Let $f$ be any function of $x$, $y$, and $z$ (among possibly other variables). Then $\pi\varpi\pi f$ equals \[\frac{[T_{xy}^{2}\widetilde{T}_{yz}(x-z)\!-\!\widetilde{T}_{xz}Q^{0}_{xy}Q^{0}_{yx}(y-z)]f\!-\!Q^{0}_{xy}[T_{xy}\widetilde{T}_{yz}(x-z)\!-\!T_{yx}\widetilde{T}_{xz}(y-z)]sf}{(x-y)^{2}(x-z)(y-z)}-\] \[-\frac{T_{xy}\widetilde{Q}^{0}_{yz}[T_{xz}\sigma f-Q^{0}_{xz}\sigma sf]-Q_{xy}^{0}\widetilde{Q}_{xz}^{0}[T_{yz}s\sigma f-Q^{0}_{yz}s\sigma sf]}{(x-y)(x-z)(y-z)},\] while the expression for $\varpi\pi\varpi f$ is given by  \[\frac{[T_{xy}\widetilde{T}_{yz}^{2}(x-z)\!-\!T_{xz}\widetilde{Q}^{0}_{yz}\widetilde{Q}^{0}_{zy}(x-y)]f\!-\!\widetilde{Q}^{0}_{yz}[T_{xy}\widetilde{T}_{yz}(x-z)\!-\!\widetilde{T}_{zy}T_{xz}(x-y)]\sigma f}{(x-y)(x-z)(y-z)^{2}}-\] \[-\frac{\widetilde{T}_{yz}Q^{0}_{xy}[\widetilde{T}_{xz}sf-\widetilde{Q}^{0}_{xz}s\sigma f]-\widetilde{Q}^{0}_{yz}Q^{0}_{xz}[\widetilde{T}_{xy}\sigma sf-\widetilde{Q}^{0}_{xy}\sigma s\sigma f]}{(x-y)(x-z)(y-z)}.\] Moreover, the two compositions coincide for every function $f$ if and only if the polynomials multiplying $f$, $sf$, $\sigma f$, $s\sigma f$, and $\sigma sf$ are the same on both sides, and the one of $s\sigma sf$ coincides with that of $\sigma s\sigma f$. \label{expansions}
\end{lem}

\begin{proof}
Recall from the proof of Corollary \ref{Tcan} that for every function $g$ we have $\pi g=\frac{T_{xy}g-Q^{0}_{xy}sf}{x-y}$ and thus also $\varpi g=\frac{\widetilde{T}_{yz}g-\widetilde{Q}^{0}_{yz}sf}{y-z}$. This expresses $\varpi\pi f$ as \[\widetilde{T}_{yz}\frac{T_{xy}f-Q^{0}_{xy}sf}{(x-y)(y-z)}-\widetilde{Q}^{0}_{yz}\frac{T_{xz}\sigma f-Q^{0}_{xz}\sigma sf}{(x-z)(y-z)},\] where in the second term the denominator is symmetric in $x$ and $y$, so that the action of $\pi$ on it yields the second asserted fraction. We divide and multiply the first term by $x-z$, yielding a denominator which is multiplied by $-1$ under the action of $s$. Applying the action of $\pi$, and gathering the multipliers of $f$ and $sf$, then produces the first term in the desired formula.

In a similar manner we get that $\pi\varpi f$ equals \[T_{xy}\frac{\widetilde{T}_{yz}f-\widetilde{Q}^{0}_{yz}\sigma f}{(x-y)(y-z)}-Q^{0}_{xy}\frac{\widetilde{T}_{xz}sf-\widetilde{Q}^{0}_{xz}s\sigma f}{(x-y)(x-z)},\] with the second denominator being $\sigma$-invariant so that the $\varpi$-image of that term is the last required term. Expanding the first fraction by $x-z$ to get a denominator that is multiplied by $-1$ under $\sigma$, letting $\varpi$ act, and gather the terms containing $f$ and $\sigma f$ yields the remaining asserted expression.

For the comparison, we recall that $\sigma s\sigma=s\sigma s$, so that if the coefficients coincide as stated then so do the operators. Conversely, take $f(x,y)$ to be $x^{2N}y^{N}$ for some very large $N$, so that its images under the combinations of $s$ and $\sigma$ are $x^{2N}z^{N}$, $y^{2N}x^{N}$, $y^{2N}z^{N}$, $z^{2N}x^{N}$, and $z^{2N}y^{N}$. When $N$ is large enough, we can recognize the contribution of $f$ to each side as the part in which the exponent of $x$ is at least $2N$ and that of $y$ is at least $N$, and similarly for the contributions of all the images of $f$ under permutations. Thus comparing both sides with such a function $f$ implies the equality of the coefficients as desired. This proves the lemma.
\end{proof}

\medskip

\begin{rmk}
When we make the comparisons of the coefficients below, we will have many instances of a common multiplier that is some $Q_{0}$ or $T$. These multipliers can be cancelled out only when they are non-zero. Note, however, that if $Q_{0}=0$ then our operator $\pi$ simply multiplies $f$ by $R_{0}$ (see Equation \eqref{poldifop} and Proposition \ref{poldifop}), and when $T$ vanishes, Equation \eqref{Tpol} with $Q_{0}$ and $R_{0}$ (and $P=0$ yields $Q_{0}(x,y)=-(x-y)R_{0}(x,y)$, and $\pi f$ reduces to $R_{0} \cdot sf$ (see, e.g., the proof of Corollary \ref{Tcan}). We thus call the operator $\pi$ \emph{degenerate} if $Q_{0}=0$ or $T=0$, and \emph{non-degenerate} otherwise. The 0 operator is, of course, degenerate, but we exclude it from our calculations as it breaks the family $\{\pi_{i}\}_{i=1}^{n-1}$ of polynomial divided difference operators into one or more families that are associated with smaller permutation groups. \label{degnondeg}
\end{rmk}

Some of the polynomials that we shall encounter below will be related by the following property.
\begin{defn}
Two non-zero polynomials $Q$ and $\widetilde{Q}$ in two variables $x$ and $y$ (among possibly others) are called \emph{almost equal} if there exist a polynomial $\widehat{Q}$ in two variables, and four univariate polynomials $q_{l}$, $q_{r}$, $\widetilde{q}_{l}$, and $\widetilde{q}_{r}$ that satisfy the \emph{product property} $q_{l}q_{r}=\widetilde{q}_{l}\widetilde{q}_{r}$, such that $Q(x,y)=q_{l}(x)q_{r}(y)\widehat{Q}(x,y)$ and $\widetilde{Q}(x,y)=\widetilde{q}_{l}(x)\widetilde{q}_{r}(y)\widehat{Q}(x,y)$. In case more variables are involved, we say that $Q$ and $\widetilde{Q}$ are \emph{almost equal in $x$ and $y$} in case confusion in terms of the variables involved may arise. \label{almosteq}
\end{defn}
Note that the condition from Definition \ref{almosteq} is equivalent to the existence of such (non-zero) $q_{l}$, $q_{r}$, $\widetilde{q}_{l}$, and $\widetilde{q}_{r}$ (with the product property) such that the equality $\widetilde{q}_{l}(x)\widetilde{q}_{r}(y)Q(x,y)=q_{l}(x)q_{r}(y)\widetilde{Q}(x,y)$ holds (and almost equality is an equivalence relation). This is easily verified by using the UFD property of the ring of polynomials. We also observe that among the irreducible polynomials in that ring, some may depend both on $x$ and on $y$, and others may depend only on one of them (ignoring the ones that depend on neither). We shall call the first ones \emph{pure} irreducible polynomials (or irreducible polynomials that are \emph{pure in $x$ and $y$} when there may be confusion as to the variables considered), and the second ones \emph{univariate} irreducible polynomials. Finally, we remark that the decomposition in that definition is not unique, as if $\widehat{Q}$ has univariate divisors then they can be put into $q_{l}$ and $\widetilde{q}_{l}$ or $q_{r}$ and $\widetilde{q}_{r}$ instead. Conversely, any common divisor of $q_{l}$ and $\widetilde{q}_{l}$ and any one of $q_{r}$ and $\widetilde{q}_{r}$ can be canceled out from them and rather put into $\widehat{Q}$.

We begin with the following observation, which will be useful for both the degenerate and the non-degenerate cases, as defined in Remark \ref{degnondeg}.
\begin{lem}
Let $\pi$ and $\varpi$ be non-zero polynomial divided difference operators as in Lemma \ref{expansions} for which the cubic braid relation holds, and assume that $Q_{0}$ and $\widetilde{Q}_{0}$ do not vanish. Then the polynomials $T$ and $\widetilde{T}$ coincide, and $Q_{0}$ and $\widetilde{Q}_{0}$ are almost equal, as in Definition \ref{almosteq}. \label{Qneq0}
\end{lem}

\begin{proof}
We note either that the coefficients of $s\sigma f$ on the two sides of Lemma \ref{expansions} are $Q^{0}_{xy}\widetilde{Q}^{0}_{xz}T_{yz}$ and $Q^{0}_{xy}\widetilde{Q}^{0}_{xz}\widetilde{T}_{yz}$, or that those of $\sigma sf$ are $Q^{0}_{xz}\widetilde{Q}^{0}_{yz}\widetilde{T}_{xy}$ and $Q^{0}_{xz}\widetilde{Q}^{0}_{yz}\widetilde{T}_{xy}$. Since the $Q_{0}$-polynomials are the same on both sides and are assumed not to vanish, the first assertion follows.

For the second one, note that comparing the coefficients of $s\sigma sf=\sigma s\sigma f$ yields the equality $Q^{0}_{xy}\widetilde{Q}^{0}_{xz}Q^{0}_{yz}=\widetilde{Q}^{0}_{xy}Q^{0}_{xz}\widetilde{Q}^{0}_{yz}$, with none of the multipliers vanishing. Consider now a pure irreducible polynomial $H$ in the variables $x$ and $y$ (for which we write $H_{xy}$ for $H(x,y)$ and similarly for its values in other pairs of variables). By purity, the polynomials $H_{xy}$, $H_{xz}$, and $H_{yz}$ are all distinct irreducible polynomials, and every polynomial that one of them divides must involve both variables showing up in it. Assume that $H$ divides $Q_{0}$, and then $H_{xz}$ divides $Q^{0}_{xz}$ and thus it divides the right hand side, hence it divides the left hand side as well. As it cannot divide $Q_{xy}$ or $Q_{yz}$ (as neither involve both $x$ and $z$), it divides $\widetilde{Q}^{0}_{xz}$ and thus $H$ divides $\widetilde{Q}_{0}$. Similarly if such $H$ divides $\widetilde{Q}_{0}$ then it divides $Q_{0}$, and by canceling out $H_{xy}H_{xz}H_{yz}$ for such $H$ from both sides and applying the same argument until no pure irreducible polynomial divides $Q_{0}$ and $\widetilde{Q}_{0}$, we deduce that the parts of $Q_{0}$ and $\widetilde{Q}_{0}$ that are based on produces of powers of pure irreducible polynomials are the same.

The remaining irreducible polynomials dividing $Q_{0}$ and $\widetilde{Q}_{0}$ are thus univariate (or involve neither $x$ nor $y$ nor $z$), so that by gathering the product of the pure irreducible polynomials into a polynomial $\widehat{Q}$ and the univariate ones as well, we obtain non-zero univariate polynomials $q_{l}$, $q_{r}$, $\widetilde{q}_{l}$, and $\widetilde{q}_{r}$ such that the presentation in Definition \ref{almosteq} holds (the irreducible polynomials that depend on neither variable can be put arbitrarily into the univariate products). Substituting these into the equality $Q^{0}_{xy}\widetilde{Q}^{0}_{xz}Q^{0}_{yz}=\widetilde{Q}^{0}_{xy}Q^{0}_{xz}\widetilde{Q}^{0}_{yz}$, and noting that $\widehat{Q}_{xy}\widehat{Q}_{xz}\widehat{Q}_{yz}$ cancels from both sides as do $q_{l}(x)\widetilde{q}_{l}(x)$ and $q_{r}(z)\widetilde{q}_{r}(z)$, we obtain the equality $q_{l}(y)q_{r}(y)=\widetilde{q}_{l}(y)\widetilde{q}_{r}(y)$, establishing the product property as well. This proves the lemma.
\end{proof}
It is clear from the proof of Lemma \ref{Qneq0} that if $\widetilde{T}=T$ and $Q_{0}$ and $\widetilde{Q}_{0}$ are almost equal, then the coefficients multiplying $s\sigma f$, $\sigma sf$, and $s\sigma sf=\sigma s\sigma f$ on both sides of Lemma \ref{expansions} are the same.

Using Lemma \ref{Qneq0}, we can immediately deduce the form involving one type of degenerate polynomial divided difference operator.
\begin{prop}
Let $\{\pi_{i}\}_{i=1}^{n-1}$ be a family of non-zero polynomial divided difference operators that satisfy the cubic braid relations. If the polynomial $T$ associated with one of these operators vanishes then it does for all of them. In this case each $\pi_{i}f$ equals $R_{i} \cdot s_{i}f$ for a polynomial $R_{i}$, such that $R_{i}$ and $R_{i+1}$ are almost equal in the variables $x_{i}$ and $x_{i+1}$ for every $1 \leq i \leq n-2$. \label{degenT}
\end{prop}
Note that the polynomials in Equation \eqref{poldifop}, and thus all those determined in Proposition \ref{canform} and Equation \eqref{Tpol}, may depend on the variables $x_{i}$ and $x_{i+1}$, as well as additional variables that are not any of the $x_{j}$'s, but not on any $x_{j}$ for $j\not\in\{i,i+1\}$. Therefore there is no problem in combining the transitivity of the equivalence relation from Definition \ref{almosteq} with changing the indices of the variables, as long as it is done on both sides of the same comparison in a compatible manner.

\begin{proof}
Consider $\pi=\pi_{i}$ and $\varpi=\pi_{i+1}$ as in Lemma \ref{expansions}. If $T=0$ then $sf$, $\sigma f$, $s\sigma f$, and $\sigma sf$ do not show up at all in the expression for $\pi\varpi\pi f$, while the one for $\varpi\pi\varpi f$ contains the term $\widetilde{T}_{yz}Q^{0}_{xy}[\widetilde{T}_{xz}sf-\widetilde{Q}^{0}_{xz}s\sigma f]$. Since we have $Q^{0}_{xy}\neq0$ and either $\widetilde{T}_{xz}\neq0$ or $\widetilde{Q}^{0}_{xz}\neq0$ by non-vanishing, it follows that $\widetilde{T}=0$. Conversely, the vanishing of $\widetilde{T}$ implies that $\varpi\pi\varpi f$ does not involve $sf$, $\sigma f$, $s\sigma f$, and $\sigma sf$, and we argue similarly using the term $T_{xy}\widetilde{Q}^{0}_{yz}[T_{xz}\sigma f-Q^{0}_{xz}\sigma sf]$ from $\pi\varpi\pi f$. By applying this argument to every $1 \leq i \leq n-1$, the first assertion is established.

Now, we recall from Equation \eqref{Tpol}, with $P=0$, $Q=Q_{0}$, and $R=R_{0}$ as in Proposition \ref{canform}, that if $T=0$ then $Q_{0}(x_{i},x_{i+1})$ reduces to $(x_{i+1}-x_{i})R_{0}(x_{i},x_{i+1})$. Moreover, the non-vanishing of all the operators implies that none of the $Q_{0}$ polynomials vanishes, and to put the dependence of $i$ back into the notation we write $Q_{i}$ and $R_{i}$ for the polynomials $Q_{0}$ and $R_{0}$ that are associated with $\pi_{i}$, so that $Q_{i}=(x_{i+1}-x_{i})R_{i}$. The formula from the proof of Corollary \ref{Tcan} then expresses $\pi_{i}f$ as $R_{i} \cdot s_{i}f$, as asserted.

But Lemma \ref{Qneq0} implies, via non-vanishing, that $Q_{i}$ and $Q_{i+1}$ are almost equal in the variables $x_{i}$ and $x_{i+1}$ for every $1 \leq i \leq n-2$. As the polynomial $x_{i+1}-x_{i}$ is a pure irreducible polynomial that divides $Q_{i}$, we can cancel it from the corresponding $\widehat{Q}$-part in Definition \ref{almosteq}, and obtain the same property for the $R_{i}$'s (and as already mentioned, moving the indices of the variables compatibly does not interfere with this argument). This proves the proposition.
\end{proof}
The operators from Proposition \ref{degenT} are those obtained from Equation \eqref{poldifop} by the substitution $P=Q=R=0$, since the proof of Proposition \ref{canform} shows that in this case we have $R_{0}=S$ and $Q_{0}(x_{i},x_{i+1})=(x_{i+1}-x_{i})S(x_{i},x_{i+1})$. Proposition \ref{degenT} states, in fact, that when a family of non-zero polynomial divided difference operators satisfying the cubic braid relations involves such an operator, then all the operators are of this sort, and the different $S$-polynomials are almost equal.

\section{The Non-Degenerate Operators \label{NonDeg}}

Since the other degenerate cases eventually lead to more complicated results, we now turn to the families of non-degenerate polynomial divided difference operators. Recall from Remark \ref{degnondeg} that this means that for every operator $\pi_{i}$ in the family, neither the polynomial $Q_{0}$ from Proposition \ref{canform} nor the polynomial $T$ from Equation \eqref{Tpol} vanishes.

We adopt again the notation with $x$, $y$, $z$, $s$, $\sigma$, $\partial$, $\delta$, $\pi$, $\varpi$, etc., where now we have, by Lemma \ref{Qneq0}, the same polynomial $T$ for all the operators. We shall keep using the index notation like $T_{xy}$ for $T(x,y)$, and write the polynomials associated with $\pi$ and $\varpi$, as $Q_{0}$ and $\widetilde{Q}_{0}$ respectively (they are almost equal by Lemma \ref{Qneq0}). We now deduce from Lemma \ref{expansions} the following result.
\begin{prop}
A necessary condition for the braid relation $\pi\varpi\pi=\varpi\pi\varpi$ is, when neither $Q_{0}$ nor $\widetilde{Q}_{0}$ nor the joint polynomial $T$ vanishes, that the polynomial $T$ is of the form $T_{xy}=axy+bx+cy+d$ for constants $a$, $b$, $c$, and $d$. \label{sfsigmaf}
\end{prop}

\begin{proof}
Note that $sf$ shows up on both expressions in Lemma \ref{expansions} with the multiplier $Q^{0}_{xy}$, and for $\sigma f$ we have a multiplier of $\widetilde{Q}^{0}_{yz}$. By removing denominators in both comparisons, canceling these multipliers (which are non-zero by non-degeneracy), and moving sides, we get the equalities
\begin{equation}
T_{xz}[T_{yx}(y-z)+T_{yz}(x-y)]=T_{xy}T_{yz}(x-z)=T_{xz}[T_{zy}(x-y)+T_{xy}(y-z)]. \label{relswithT}
\end{equation}
The equality between the two extremal expressions in Equation \eqref{relswithT} reduces, after canceling $T_{xz}$ (which is also non-zero by non-degeneracy), to the equality $(T_{xy}-T_{yx})(y-z)=(T_{yz}-T_{zy})(x-y)$, which amounts to $\partial T_{xy}=\delta T_{yz}$ after dividing both sides by $(x-y)(y-z)$. But as the first expression is symmetric in $x$ and $y$ and independent of $z$, and the second one is symmetric in $y$ and $z$ and independent of $x$, we deduce that this equality can hold only if both sides equal a constant, which we denote by $\mu$. It follows, via the proof of Lemma \ref{symimdi}, that $T_{xy}=T(x,y)$ must be of the form $\mu x+V(x,y)$ where $V$ is a symmetric polynomial in two variables.

We now claim that the variable degree $r$ of the symmetric polynomial $V$, defined in Remark \ref{vardeg}, is at most 1. Indeed, if $r>1$ then the degree in $y$ (ignoring powers of $x$ and $z$) of $T_{xy}$, $T_{yx}$, $T_{yz}$, and $T_{zy}$ is $r$, while it is 1 for $x-y$ and $y-z$ and 0 for $T_{xz}$. But then the degree in $y$ of the extremal expressions in Equation \eqref{relswithT} is $r+1$, while it is $2r>r+1$ for the middle one and this equation cannot hold. This proves the claim, from which it follows, via the symmetry of $V$, that $V(x,y)$ must be of the form $axy+c(x+y)+d$ where $a$, $c$, and $d$ are constants. By adding $\mu x$ and setting $b:=c+\mu$, the desired formula for $T_{xy}$ follows. This completes the proof of the proposition.
\end{proof}

For obtaining the next property, we shall need the following technical result, for which we recall that $\lceil t \rceil$ is the minimal integer that is not smaller than the real number $t$, and $\lfloor t \rfloor$ is the maximal integer that is not larger than $t$.
\begin{lem}
Let $B$ be a symmetric polynomial in two variables $x$ and $y$ that is homogeneous of degree $m$, and let $r$ be its variable degree. Then the only symmetric polynomials $A$ that are homogeneous of degree $m$ and such that the variable degree of $A(A+B)$ is less than $r+\lceil m/2 \rceil$ are $a=0$ and $A=-B$, with vanishing product. In particular, if $C$ is a non-zero symmetric polynomial that is homogeneous of degree $2r$ and has variable degree less than $r+\lceil m/2 \rceil$ then the equation $A(A+B)+C=0$ has no solution $A$. \label{vanprod}
\end{lem}

\begin{proof}
If we write the homogeneity degree $m$ as $2k+\varepsilon$ with $\varepsilon\in\{0,1\}$, then the equality $2r \geq m$ implies that $r=k+\varepsilon+h$ for some $0 \leq h \leq k$ (since $m-r$ then equals $k-h$). It follows that the $B$ is spanned by the polynomials $\big\{\sum_{l=k-j}^{k+\varepsilon+j}x^{l}y^{m-l}\big\}_{j=0}^{h}$ (since the variable degree of the $j$th polynomial here is $k+\varepsilon+j$), and if the spanning coefficients for $B$ are $\{\beta_{i}\}_{i=0}^{h}$ then $\beta_{h}\neq0$ by the assumption on the variable degree of $B$. Moreover, when looking for such a polynomial $A$, note that if $A$ has variable degree larger than $r$ then so is the variable degree of $A+B$, and thus $A(A+B)$ has variable degree that is larger than $2r$ hence than $r+\lceil m/2 \rceil$. Thus we may restrict attention to polynomials $A$ that have variable degree at most $r$, meaning that we can write $A(x,y)$ as $\sum_{j=0}^{h}\alpha_{j}\sum_{l=k-j}^{k+\varepsilon+j}x^{l}y^{r-l}$ with some coefficients $\{\alpha_{j}\}_{j=0}^{h}$ as well. The coefficients for $A+B$ are thus $\{\alpha_{i}+\beta_{i}\}_{i=0}^{h}$.

Now, the assumption on the variable degree of the product is equivalent, via the fact that $m-\lceil m/2 \rceil=\lfloor m/2 \rfloor$, to the assertion that the product is divisible by $(xy)^{m+\lfloor m/2 \rfloor-r+1}$. Using $k$ and $h$, this exponent is $2k-h+1$. This property is unaffected by multiplying by $x-y$, and note that the $j$th basis polynomial is $\partial(x^{k+\varepsilon+j+1}y^{k-j})$, or equivalently $(xy)^{k-j}\partial(x^{2j+\varepsilon+1})$, namely multiplying it by $x-y$ produces $(xy)^{k-j}(x^{2j+\varepsilon+1}-y^{2j+\varepsilon+1})$. Moreover, the product of two such polynomials, say the $i$th and the $j$th, is \[(xy)^{2k-i-j}\!\big(x^{2(i+j)+\varepsilon+1}+y^{2(i+j)+\varepsilon+1}\!-(xy)^{2\min\{i,j\}+\varepsilon+1}(x^{2\max\{i,j\}}+y^{2\max\{i,j\}})\big)\!,\] and as in the last terms the total exponent is $2k-|i-j|+\varepsilon+1$, which is at least $2k-h+1$ when $i$ and $j$ are between 0 and $k$, we may ignore these terms in our calculations.

It follows that the divisibility of $A(A+B)$ by $(xy)^{2k-h+1}$ is equivalent to the latter monomial dividing the expression \[\textstyle{\sum_{i=0}^{h}\sum_{j=0}^{j}\alpha_{j}(\alpha_{i}+\beta_{i})(xy)^{2k-i-j}(x^{2(i+j)+\varepsilon+1}+y^{2(i+j)+\varepsilon+1})},\] where the expression corresponding to $i$ and $j$ depends only on $i$ and $j$ and is divisible by precisely $(xy)^{2k-i-j}$ and no higher power of $xy$. It follows that the total coefficient that expression must vanish whenever $i+j \geq h$ (the terms with $i+j>h$ are divisible by the monomial in question. This implies, with $i+j=2h-p$, the equality $\sum_{i=h-p}^{h}\alpha_{2h-p-i}(\alpha_{i}+\beta_{i})=0$ for every $0 \leq p \leq h$.

Now, for $p=0$ this is $\alpha_{h}(\alpha_{h}+\beta_{h})=0$, meaning that either $\alpha_{h}=0$ or $\alpha_{h}+\beta_{h}=0$ (but not both, since $\beta_{h}\neq0$ by assumption). We prove, by decreasing induction on $l$, that $\alpha_{l}=0$ in the first case and $\alpha_{l}+\beta_{l}=0$ in the second one, where we have established the basis of the induction with $l=h$. Assume that $l<h$ and that the assertion holds for every index $l<j \leq h$, and consider the equality with $p=h-l$. In the first case we have $\alpha_{2h-p-i}=\alpha_{h+l-i}=0$ for every $h-l \leq i<h$ by the induction hypothesis, and the last summand, with $i=h$, is $\alpha_{l}$ times the non-zero expression $\alpha_{h}+\beta_{h}=\beta_{h}$, implying that $\alpha_{l}=0$ as desired. In the second case the induction hypothesis yields $\alpha_{i}+\beta_{i}=0$ for every $l<i \leq h$, with the remaining summand, associated with $i=l$, being $\alpha_{l}+\beta_{l}$ times $\alpha_{h}=-\beta_{h}\neq0$, so that $\alpha_{l}+\beta_{l}$ as needed. This proves our claim.

But then the first case yields $A=0$, and the second one produces $A=-B$, which establishes the first assertion. The second one thus follows from the first, since we then get $A(A+B)=0$ and thus $A(A+B)+C=C\neq0$ by the assumption on $C$. This completes the proof of the lemma.
\end{proof}

We can now deduce another property that is required for our relation.
\begin{prop}
For $T$ as in Proposition \ref{sfsigmaf}, and with $R_{+}$ and $\widetilde{R}_{+}$ that are $\partial$-positive (or $\delta$-positive), if the cubic braid relation holds then $R^{+}_{xy}$ must be just a multiple of $x$, $\widetilde{R}^{+}_{yz}$ has to be a multiple of $y$, and $\partial P_{xy}$ and $\delta\widetilde{P}_{yz}$ have degree at most 1. \label{lastcoeff}
\end{prop}

\begin{proof}
Consider the coefficients multiplying $f$ in Lemma \ref{expansions}, with $\widetilde{T}=T$ by Lemma \ref{Qneq0}, and recall from the proof of Proposition \ref{sfsigmaf} that the cubic braid relation implies Equation \eqref{relswithT}. As this equation expresses $T_{xy}^{2}T_{yz}(x-z)$ as $T_{xy}T_{xz}[T_{yx}(y-z)+T_{yz}(x-y)]$ and $T_{xy}T_{yz}^{2}(x-z)$ as $T_{xz}T_{yz}[T_{zy}(x-y)+T_{xy}(y-z)]$, we can divide again by the common multiplier $T_{xz}$ (by non-degeneracy), cancel $T_{xy}T_{yz}$ from both sides, and obtain the equality \[[T_{xy}T_{yx}-Q^{0}_{xy}Q^{0}_{yx}]\tfrac{y-z}{x-y}=[T_{yz}T_{zy}-\widetilde{Q}^{0}_{yz}\widetilde{Q}^{0}_{zy}]\tfrac{x-y}{y-z}\] (in fact, since $Q_{0}$ and $\widetilde{Q}_{0}$ are almost equal, the product $\widetilde{Q}^{0}_{yz}\widetilde{Q}^{0}_{zy}$ from above equals $Q^{0}_{yz}Q^{0}_{zy}$ and one can equally well work with $Q$ alone below).

For applying Proposition \ref{Hecke} below, we first recall from Equation \eqref{Tpol} and Remark \ref{degnondeg} that $Q^{0}_{xy}$ can be written as $T_{xy}-(x-y)R^{0}_{xy}$, which means that $Q^{0}_{yx}=T_{yx}+(x-y)R^{0}_{yx}$, and similarly for $\widetilde{Q}^{0}_{yz}$ and $\widetilde{Q}^{0}_{zy}$. The left hand side thus becomes $(x-y)(y-z)$ times $R^{0}_{xy}R^{0}_{yx}+(R^{0}_{xy}T_{yx}-R^{0}_{yx}T_{xy})/(x-y)$, with the second summand being just $\partial(R^{0}_{xy}T_{yx})$, and the right hand side is thus the same multiplier times $\widetilde{R}^{0}_{yz}\widetilde{R}^{0}_{zy}+\delta(\widetilde{R}^{0}_{yz}T_{zy})$. We cancel this multiplier, and are again left with an equality between a polynomial that is symmetric in $x$ and $y$ and independent of $z$ and a polynomial that is symmetric in $y$ and $z$ and independent of $x$. Thus, as in the proof of Proposition \ref{sfsigmaf}, both polynomials are equal to the same constant $\nu$.

We thus concentrate on the equality $\partial(R^{0}_{xy}T_{yx})+R^{0}_{xy}R^{0}_{yx}=\nu$ and its consequences, and the equality involving $\widetilde{R}^{0}_{yz}$ and $T_{zy}$ will yield similar consequences in the same manner. We recall from the proof of Proposition \ref{canform} that if $R=R_{+}$ then $R_{0}=R_{+}+\partial P$, and with $T$ as in Proposition \ref{sfsigmaf}, we can write $T_{yx}$ as the sum of the symmetric expression $axy+bx+by+d$ and $(c-b)x$. Our second summand thus expands as $R^{+}_{xy}R^{+}_{yx}+\partial P_{xy}(R^{+}_{xy}+R^{+}_{yx}+\partial P_{xy})$, Lemma \ref{Leibniz} shows that the symmetric part of $T$ multiplies $\partial R_{+}$, and the constant $c-b$ multiplies the sum of $\partial(xR_{+})$ and $\partial P$ (Lemma \ref{Leibniz} again, with the symmetric $\partial P$ and with $\partial x=1$). Altogether we obtain the equality
\begin{equation}
(axy+bx+by+d)\partial R_{+}+R_{+} \cdot sR_{+}+(c-b)\partial(Q+xR_{+})+\partial Q\cdot[(R_{+}+sR_{+})+\partial Q]=\nu. \label{RanddQ}
\end{equation}

Denote the degree of $R_{+}$ by $m$, and recall that if $R_{+}\neq0$ then $m\geq1$ since $R_{+}$ is $\partial$-positive. The $\partial$-positivity implies that $R_{+}+sR_{+}$ is also of degree $m$ and the degree of $\partial R_{+}$ is $m-1$, and write $h$ for the degree of $\partial Q$. Then the degrees of the terms on the left hand side of Equation \eqref{RanddQ} are $m+1$, $2m$, $h$, $m$, $m+h$ (for $\partial Q\cdot(R_{+}+sR_{+})$ considered as a single term), and $2h$ respectively. As the total combination equals the constant $\nu$ by that equation, it follows that $h$ cannot be larger than $m$ (since then the highest degree part of the term $(\partial Q)^{2}$, of degree $2h$, cannot cancel), and if $m>1$, so that $2m>m+1$, the degree $h$ cannot be smaller than $m$ either (by applying the same consideration to the highest degree part of $R_{+} \cdot sR_{+}$, of degree $2m$).

The last paragraph assumed that $R_{+}\neq0$. But when $R_{+}=0$ Equation \eqref{RanddQ} only involves the terms $(c-b)\partial Q$, of degree $h$ when non-vanishing, and $(\partial Q)^{2}$, of degree $2h$, on the left hand side. Thus similar considerations show that if $R_{+}=0$ then $h\leq0$, meaning that $\partial Q$ is a constant in this case.

It follows that if $m>1$ then we must have $h=m$, and we set $B$ to be the part of $R_{+}+sR_{+}$ that is homogenous of maximal degree $m$. If $r$ be the variable degree of $R_{+}+sR_{+}$, then the fact that $R_{+}$, hence also its $m$-homogenous part, is $\partial$-positive, implies that the degree of this part of $R_{+}$ in the variable $x$ is also $r$, while that of the $m$-homogenous part $sR_{+}$ is smaller than $\lceil m/2 \rceil$ (for the exponent of $y$ to always be larger). Therefore if $C$ is the part of the symmetric polynomial $R_{+} \cdot sR_{+}$ that is homogenous of degree $2m$ then it is the product of these parts of $R_{+}$ and $sR_{+}$, and its degree in $x$, which is its variable degree, is smaller than $r+\lceil m/2 \rceil$. Take now $A$ to be the part of $\partial Q$ that is homogenous of maximal degree $m$, and then our degree considerations show that the left hand side of Equation \eqref{RanddQ} is $A(A+B)+C$. But this equation implies the vanishing of the latter expression and we have $C\neq0$ by definition, while Lemma \ref{vanprod} shows that there is no choice for $A$ such that this equality holds under these assumptions.

We have thus proved that the degree $m$ of $R_{+}$ is bounded by 1, meaning that $R^{+}_{xy}$ is a multiple of $x$ since it is $\partial$-positive. Moreover, we saw that the degree of $\partial Q$ is bounded by $m$ when $m\geq1$ and by 0 when $R_{+}=0$, so that it is also bounded by 1 in general. From the analogue of Equation \eqref{RanddQ} with $\widetilde{R}_{+}=\widetilde{R}^{+}_{yz}$, $\widetilde{Q}=\widetilde{Q}_{yz}$, and $\delta\widetilde{Q}=\delta\widetilde{Q}_{yz}$, we obtain that the former is a multiple of $y$ and the degree of the latter is also bounded by 1. This proves the proposition.
\end{proof}

\medskip

Propositions \ref{sfsigmaf} and \ref{lastcoeff} put restrictions on the polynomials $T$, $R$, and $\partial Q_{xy}$ that arise from the fact that the cubic braid relation in question implies the validity of Equations \eqref{relswithT} and \eqref{RanddQ}. But we can extract more data out of these equations. We begin with the former one.
\begin{lem}
Let $T$ be as in Proposition \ref{sfsigmaf}. The Equation \eqref{relswithT} holds if and only if $ad=bc$. \label{coeffT}
\end{lem}

\begin{proof}
It is straightforward to verify, with the given expression for $T$, that both terms in brackets in Equation \eqref{relswithT} reduce to $(ay^{2}+by+cy+d)(x-z)$. We can thus cancel the multiplier $x-z$, so that this equation is equivalent to the equality $T_{xy}T_{yz}-(ay^{2}+by+cy+d)T_{xz}=0$. But by substituting all the $T$-polynomials, this difference is just $(ad-bc)(x-y)(y-z)$, and the vanishing of this polynomial is equivalent to the vanishing of the asserted expression. This proves the lemma.
\end{proof}

The second equation produces the following extra equality.
\begin{lem}
Take some constants $a$, $b$, $c$, and $d$, set $R_{+}(x,y)$ to be $\alpha x$ for another constant $\alpha$, let $\beta$ and $e$ be two additional constants, and let $Q$ be a polynomial in two variables such that $\partial Q_{xy}=\beta(x+y)+e$. Then Equation \eqref{RanddQ} holds if and only if either $\alpha=\beta=0$, or $\alpha=-a$, $\beta=0$, and $e=-c$, or $\alpha=-a$, $\beta=a$, and $e=b$. \label{restonQR}
\end{lem}

\begin{proof}
With our $R_{+}$ the expressions $\partial R_{+}$, $R_{+} \cdot sR_{+}$, $\partial(xR_{+})$, and $R_{+}+sR_{+}$ are $\alpha$, $\alpha^{2}xy$, $\alpha(x+y)$, and $\alpha(x+y)$ respectively. Substituting these, as well as the value of $\partial Q$, into Equation \eqref{RanddQ} produces \[\alpha(axy+bx+by+d)+\alpha^{2}xy+[\beta(x+y)+e+c-b][(\alpha+\beta)(x+y)+e]=\nu.\] The left hand side is the sum of the quadratic terms $\beta(\alpha+\beta)(x+y)^{2}$ and $\alpha(a+\alpha)xy$, the linear term $[b\alpha+(\alpha+\beta)(c-b)+e(\alpha+2\beta)](x+y)$, and the constant term $d\alpha+e(e+c-b)$. Thus the equality holds if and only if the first three terms, namely their coefficients, vanish.

Now, the vanishing of $\alpha(a+\alpha)$ implies that either $\alpha=0$ or $\alpha=-a$. When $\alpha=0$, the vanishing of $\beta(\alpha+\beta)$ implies that $\beta=0$ as well, and then the linear term already vanishes regardless of the value of $e$. Otherwise $\alpha=-a$, and from $\beta(\alpha+\beta)$ we get either $\beta=0$ or $\beta=a$. When $\beta=0$, the remaining expression to vanish is $-a(c+e)$, which implies that either $a=0$ and we are in the previous case again, or $e=-c$. The remaining case is where $\beta=a$, yielding the vanishing of $a(e-b)$, hence either $a=0$ which sends us to the previous case yet again, or we have $e=b$. This proves the lemma.
\end{proof}

\begin{rmk}
Note that the constant $\nu$ equals, by the proof of Lemma \ref{restonQR}, to the constant $d\alpha+e(e+c-b)$. In the first case there it is just $e(e+c-b)$, while in the other two cases it becomes $-ad+bc$, which vanishes when Lemma \ref{coeffT} holds. The fact that if $\alpha=0$ then $\beta=0$ in Lemma \ref{restonQR} is in correspondence with the fact that when $R_{+}$ vanished in the proof of Proposition \ref{lastcoeff}, the degree of $\partial Q$ was bounded by 0 rather than 1. Of course, Lemma \ref{restonQR} holds equally well for $\widetilde{R}_{+}(y,z)=\widetilde{\alpha}y$ and $\delta\widetilde{Q}(y,z)=\widetilde{\beta}(y+z)+\widetilde{e}$. \label{valofnu}
\end{rmk}

\medskip

We can now prove our main theorem.
\begin{thm}
Let $\{\pi_{i}\}_{i=1}^{n-1}$ be a family of polynomial divided difference operators such that for any $1 \leq i \leq n-1$, neither the polynomial $Q_{0}$ from Proposition \ref{canform} nor the polynomial $T$ from Equation \eqref{Tpol} vanishes. Then the cubic braid relations hold between these operators if and only if the family is as in one of the following two cases.
\begin{enumerate}[(1)]
\item There exist constants $a$, $b$, $c$, $d$, not all 0 and such that $ad-bc=0$, and another constant $e$ that equals neither 0 nor $b-c$, such that the polynomials Proposition \ref{canform} are $P^{+}(x_{i},x_{i+1})=ax_{i}x_{i+1}+(b-e)x_{i}+cx_{i+1}+d$, $Q_{+}(x_{i},x_{i+1})=ex_{i}$, and $R_{+}(x_{i},x_{i+1})=0$, uniformly for all $1 \leq i \leq n-1$.
\item There are $a$, $b$, $c$, and $d$ with the same properties, such that for every $i$ independently the polynomials $P^{+}(x_{i},x_{i+1})$, $Q_{+}(x_{i},x_{i+1})$, and $R_{+}(x_{i},x_{i+1})$ are given by \[\begin{matrix} ax_{i}x_{i+1}+bx_{i}+cx_{i+1}+d, & 0, & 0; \\ ax_{i}x_{i+1}+cx_{i}+cx_{i+1}+d, & (b-c)x_{i}, & 0; \\ ax_{i}^{2}+(b+c)x_{i}+cx_{i+1}+d, & -cx_{i}, & -ax_{i}; \\ cx_{i+1}+d, & ax_{i}^{2}+bx_{i}, & -ax_{i}. \end{matrix}\]
\end{enumerate} \label{main}
\end{thm}

\begin{proof}
The asserted relation holds if and only if the two expressions from Lemma \ref{expansions} are equal. We saw that in the non-degenerate case, the conditions from Lemma \ref{Qneq0} are equivalent to the coefficients of $\sigma sf$, $s\sigma f$, and $s\sigma sf=\sigma s\sigma f$ in both of them being the same. Moreover, under these conditions, the proof of Proposition \ref{sfsigmaf} shows that the coefficients of $sf$ and $\sigma f$ coincide if and only if Equation \eqref{relswithT} holds. From Lemma \ref{coeffT} we deduce that this is the case precisely when the polynomial $T(x,y)$ from Equation \eqref{Tpol} equals $axy+bx+cy+d$ for constants $a$, $b$, $c$, and $d$ that satisfy $ad=bc$. In addition, under these assumptions, it follows from the proof of Proposition \ref{lastcoeff} that the coefficients of $f$ on both sides are also equal precisely when Equation \eqref{RanddQ} is satisfied, both for $Q$ and $R_{+}$ and for $\widetilde{Q}$ and $\widetilde{R}_{+}$.

Next, assume that $Q=Q_{+}$ is $\partial$-positive and $\widetilde{Q}=\widetilde{Q}_{+}$ is $\delta$-positive, and then the proposition itself, together with the uniqueness in Lemma \ref{symimdi}, imply that $R_{+}(x,y)=\alpha x$, $Q_{+}(x,y)=\beta x^{2}+ex$, $\widetilde{R}_{+}(y,z)=\widetilde{\alpha}y$, and $\widetilde{Q}_{+}(y,z)=\widetilde{\beta}y^{2}+\widetilde{e}y$ for constants $\alpha$, $\beta$, $e$, $\widetilde{\alpha}$, $\widetilde{\beta}$, $\widetilde{e}$. Recalling from Equation \eqref{Tpol} and Proposition \ref{canform} (with $S=0$) that $P^{+}(x,y)=T(x,y)-Q_{+}(x,y)-(x-y)R_{+}(x,y)$ and $Q_{0}(x,y)=P^{+}(y,x)+Q_{+}(x,y)$, we deduce that \[Q_{0}(x,y)=\beta x^{2}+(a+\alpha)xy-(\alpha+\beta)y^{2}+(c+e)x+(b-e)y+d,\] and similarly for the polynomial that is associated with $\varpi$ we get \[\widetilde{Q}_{0}(y,z)=\widetilde{\beta}y^{2}+(a+\widetilde{\alpha})yz-(\widetilde{\alpha}+\widetilde{\beta})y^{2}+(c+\widetilde{e})y+(b-\widetilde{e})z+d.\] We also recall from Lemma \ref{Qneq0} that $Q_{0}$ and $\widetilde{Q}_{0}$ are almost equal, in the sense of Definition \ref{almosteq}.

Now, Lemma \ref{restonQR} determines the cases of values of the constants that satisfy Equation \eqref{RanddQ}, and we consider first the situation with $\alpha=\beta=0$. Then the expression for $Q_{0}$ is either univariate or reducible if and only if $ad=(b-e)(c+e)$, which is equivalent to $e(e+c-b)=0$ since $ad=bc$ by assumption. Thus for $e$ that equals neither 0 nor $b-c$, the polynomial $Q_{0}$ is pure and irreducible, so that is equals $\widetilde{Q}_{0}$ as well and hence $\widetilde{\alpha}=\widetilde{\beta}=0$ and $\widetilde{e}=e$. Clearly, when $\widetilde{\alpha}=\widetilde{\beta}=0$ and $\widetilde{e}$ is neither 0 nor $b-c$ we get a pure irreducible $\widetilde{Q}_{0}$ and thus $Q_{0}=\widetilde{Q}_{0}$, $\alpha=\beta=0$, and $e=\widetilde{e}$. As this is the case when $\pi=\pi_{i}$ and $\varpi=\pi_{i+1}$ for any $1 \leq i \leq n-2$, we deduce that the polynomials are independent of $i$, and this produces, by using the original variables, the first asserted case.

It remains to consider the cases with $e=0$ and with $e=b-c$, as well as the two other possibilities from Lemma \ref{restonQR}. This gives the values $axy+bx+cy+d$, $axy+cx+by+d$, $ay^{2}+(b+c)y+d$, and $ax^{2}+(b+c)x+d$ for $Q_{0}(x,y)$, and since $ad=bc$, they are all almost equal (either they are all the same linear polynomial in either $x$ or $y$ when $a=0$, or they are the product of two non-constant linear polynomials, each of which can be in either $x$ or $y$). Similarly, the resulting values of $\widetilde{Q}_{0}$ are the same (in the appropriate variables). Bt returning to the original variables again, this yields the second asserted case, and since all of our arguments are invertible, we also proved that both cases do construct families of non-degenerate polynomial divided difference operators that satisfy the cubic braid relations. This proves the theorem.
\end{proof}

\begin{rmk}
Theorem \ref{main} expresses the families of polynomial divided difference operators using the polynomials $P_{+}$, $Q^{+}$, and $R_{+}$ from Proposition \ref{canform}. Here we record the formulae for the polynomials in the other normalizations from Proposition \ref{canform}, in the two cases from that theorem.
\begin{enumerate}[(1)]
\item $P_{+}(x_{i},x_{i+1})=(b-c-e)x_{i}$ and $Q^{+}(x_{i},x_{i+1})=ax_{i}x_{i+1}+(c+e)x_{i}+cx_{i+1}+d$, and with them $Q_{0}(x_{i},x_{i+1})=ax_{i}x_{i+1}+(c+e)x_{i}+(b-e)x_{i+1}+d$ and $R_{0}(x_{i},x_{i+1})=b-c-e$, uniformly for all $1 \leq i \leq n-1$, where $e$ equals neither 0 nor $b-c$.
\item For each $i$ separately, the values of $P_{+}(x_{i},x_{i+1})$, $Q^{+}(x_{i},x_{i+1})$, $Q_{0}(x_{i},x_{i+1})$, and $R_{0}(x_{i},x_{i+1})$ are given by
\[\begin{matrix} (b-c)x_{i}, & ax_{i}x_{i+1}+cx_{i}+cx_{i+1}+d, & ax_{i}x_{i+1}+cx_{i}+bx_{i+1}+d, & b-c; \\ 0, & ax_{i}x_{i+1}+bx_{i}+cx_{i+1}+d, & ax_{i}x_{i+1}+bx_{i}+cx_{i+1}+d, & 0; \\ ax_{i}^{2}+bx_{i}, & cx_{i+1}+d, & ax_{i+1}^{2}+(b+c)x_{i+1}+d, & ax_{i+1}+b; \\ -cx_{i}, & ax_{i}^{2}+(b+c)x_{i}+cx_{i+1}+d, & ax_{i}^{2}+(b+c)x_{i}+d, & -ax_{i}-c. \end{matrix}\]
\end{enumerate}
These are all obtained, as in the proof of Proposition \ref{canform}, by decomposing $P^{+}$ as $P_{+}+P_{s}$ with the former $\partial$-positive and the latter symmetric, and using the relations $Q^{+}=Q_{+}+P^{+}-P_{+}$, $Q_{0}=sP^{+}+Q_{+}=sP_{+}+Q^{+}$ (since $S=0$), and $R_{0}=R_{+}+\partial P_{+}=R_{+}+\partial P^{+}$. In fact, the value of $Q_{0}$ was already determined in the proof of Theorem \ref{main}, and that of $R_{0}$ is also obtained, via Equation \eqref{Tpol} as $\frac{T(x_{i},x_{i+1})-Q_{0}(x_{i},x_{i+1})}{x_{i}-x_{i+1}}$, as we recall from Proposition \ref{sfsigmaf} that $T(x_{i},x_{i+1})=ax_{i}x_{i+1}+bx_{i}+cx_{i+1}+d$. \label{otherforms}
\end{rmk}

\begin{rmk}
We have a symmetry in the polynomials from Theorem \ref{main} and Remark \ref{otherforms}, of the following kind. Taking $e$ to $b-c-e$ interchanges the $P_{+}$ with $Q_{+}$ polynomials sa well as the $P^{+}$ with $Q^{+}$ polynomials, and operates as $s=s_{i}$ on $Q_{0}$. The special case of the values 0 and $b-c$ is revealed in such a symmetry involving an interchange of the first two lines of Case (2) in them. This symmetry is completed by noticing that the same occurs when the last two lines in Case (2) are interchanged in a similar manner. We also note that in some situations, some of the four options in Case (2) coincide: When $b=c$ the first two are the same, and recall that $a=0$ then either $b=0$ or $c=0$ since $ad=bc$. Then for $a=c=0$ the third line there equals the first and the fourth one is like the second, while if $a=b=0$ then the third one coincides with the second and the fourth one is the same as the first. Finally, if $a=b=c=0$ then Case (2) contains just one option, which produces the family in which $\pi_{i}=d\partial_{i}$ for every $1 \leq i \leq n-1$, with some non-zero constant $d$. \label{lesselts}
\end{rmk}

\section{The Second Degenerate Type \label{WithVanQ0}}

To complete the analysis, it remains to consider the polynomial divided difference operators involving a vanishing $Q_{0}$. We stick again to the notation involving $x$, $y$, $z$, $s$, $\sigma$, $\partial$, $\delta$, $\pi$, $\varpi$, etc., and the first step now is the following one.
\begin{lem}
Assume that $\pi$ and $\varpi$, with the polynomials $Q_{0}$, $R_{0}$, $\widetilde{Q}_{0}$ and $\widetilde{R}_{0}$ from Proposition \ref{canform}, satisfy the cubic braid relation. Then if $\widetilde{Q}_{0}=0$ then $\widetilde{R}^{0}_{yz}$ is $\widetilde{r}(y)$ for a non-zero univariate polynomial $r$, and $\widetilde{T}_{yz}=(y-z)\widetilde{r}(y)$. Similarly, when $Q_{0}=0$ we get $T_{xy}=(x-y)R^{0}_{xy}$ and $R^{0}_{xy}=r(y)$, where $r$ is a non-zero univariate polynomial. \label{Q0univ}
\end{lem}

\begin{proof}
When $\widetilde{Q}_{0}=0$, the expression for both $\pi\varpi\pi f$ and $\varpi\pi\varpi f$ in Lemma \ref{expansions} involve only $f$ and $sf$, and we recall that $\widetilde{T}_{yz}=(y-z)\widetilde{R}^{0}_{yz}$. Thus, when comparing the coefficients of $f$ on both sides of that lemma, we get the equality of $\frac{T_{xy}^{2}\widetilde{R}^{0}_{yz}(y-z)(x-z)-\widetilde{R}^{0}_{xz}Q^{0}_{xy}Q^{0}_{yx}(x-z)(y-z)}{x-y}$ and $T_{xy}(\widetilde{R}^{0}_{yz})^{2}(y-z)(x-z)$. Canceling out the multiplier $(x-z)(y-z)$ and multiplying by $x-y$ yields the equality
\begin{equation}
(\widetilde{R}^{0}_{yz})^{2}T_{xy}(x-y)-\widetilde{R}^{0}_{yz}T_{xy}^{2}+\widetilde{R}^{0}_{xz}Q^{0}_{xy}Q^{0}_{yx}=0. \label{beforeR=r}
\end{equation}
Recalling that $\varpi\neq0$, we deduce that $\widetilde{T}\neq0$, which implies, by Proposition \ref{degenT}, that $T_{xy}\neq0$ as well.

Now, the only dependence of Equation \eqref{beforeR=r} on $z$ is through the $\widetilde{R}_{0}$ polynomials, and let $m$ denote the degree of this polynomial in the second variable. Then the degree of the first term in $z$ is $2m$ (since $T_{xy}\neq0$), while those of the other terms is bounded by $m$. This equality can only hold if $m=0$, namely if $\widetilde{R}_{0}(y,z)=\widetilde{r}(y)$ for a univariate polynomial $r$. This proves the first assertion.

The second assertion is similar, with $T_{xy}=(x-y)R^{0}_{xy}$ and $\widetilde{T}_{yz}\neq0$ via Proposition \ref{degenT}, the equality $(R^{0}_{xy})^{2}\widetilde{T}_{yz}(y-z)-R^{0}_{xy}\widetilde{T}_{yz}^{2}+R^{0}_{xz}\widetilde{Q}^{0}_{yz}\widetilde{Q}^{0}_{zy}=0$, and checking the degree of each term in $x$, so that $R^{0}_{xy}=r(y)$ for univariate $r$. This proves the lemma.
\end{proof}

We shall also need the following auxiliary result.
\begin{lem}
Assume that $\widetilde{r}$ is a univariate polynomial, and that $\widehat{T}$ is a polynomial of two variables such that $\widetilde{r}(x)\widehat{T}(x,y)-x$ is symmetric in $x$ and $y$. Then $\widetilde{r}$ has degree at most 1. If the polynomial $\widetilde{r}(x)^{2}\widehat{T}(x,y)^{2}-\widetilde{r}(y)\widehat{T}(x,y)(x-y)$ is also symmetric, then either $\widehat{T}$ is a constant and $r\widehat{T}$ is a monic linear polynomial, or $\widetilde{r}$ is a constant and $\widehat{T}$ is any polynomial such that $\widetilde{r}\widehat{T}(x,y)-x$ is symmetric. These conditions are also sufficient for these symmetry properties. \label{fromsym}
\end{lem}

\begin{proof}
Comparing our first expression with its $s$-image implies that $x-y$ lies, in the ring of polynomials in $x$ and $y$, in the ideal generated by $\widetilde{r}(x)$ and $\widetilde{r}(y)$. But modulo that ideal, the monomials $x^{a}y^{b}$ with $a$ and $b$ smaller than the degree of $f$ are still linearly independent. It follows that if this degree is at least 2, then $x-y$ cannot vanish modulo this ideal, hence cannot lie in this ideal. This proves the first assertion.

Assuming now that the degree of $\widetilde{r}$ is 1, so we can write $\widetilde{r}(x)$ as $a(x+b)$ for some constants $b$ and $a\neq0$, and then $a(x+b)\big(\widehat{T}(x,y)-\frac{1}{a}\big)$ yields the expression $\widetilde{r}(x)\widehat{T}(x,y)-x-b$, which was assumed to be symmetric (subtracting the constant $b$ does not affect this property). As it is divisible by $x+b$, it must also be divisible by $y+b$, meaning that $\widehat{T}(x,y)$ can be written as $(y+b)U(x,y)+\frac{1}{a}$ for some polynomial $U$. Moreover, the symmetric expression in question then becomes $a(x+b)(y+b)U(x,y)$, so that $U$ is also symmetric.

Now, our second expression is divisible by $\widehat{T}(x,y)=(y+b)U(x,y)+\frac{1}{a}$ and is symmetric, so that it is also divisible by $\widehat{T}(y,x)=(x+b)U(x,y)+\frac{1}{a}$ (using the symmetry of $U$). If $U\neq0$ then this polynomial does not divide $\widehat{T}(x,y)$, so that it has to divide the other multiplier, which is the scalar $a$ times the expression $a(x+b)^{2}\big[(y+b)U(x,y)+\frac{1}{a}]-(y+b)(x-y)$. If we subtract the multiple $a(y+b)(x+b)[(x+b)U(x,y)+\frac{1}{a}]$ of $\widehat{T}(y,x)$, then the difference, which equals $(x+b)^{2}-(y+b)(x+b)-(y+b)(x-y)=(x-y)^{2}$, is also divisible by $\widehat{T}(y,x)$. But as neither $x-y$ nor its square are of the form $(x+b)U(x,y)+\frac{1}{a}$, we deduce that $U$ cannot be non-zero. But then $U=0$, $\widehat{T}=\frac{1}{a}$, and $\widehat{T}\widetilde{r}(x)$ is the monic polynomial $x+b$, so that $\widetilde{r}(x)\widehat{T}-x$ is the (symmetric) constant $b$ and $\widetilde{r}(x)^{2}\widehat{T}^{2}-\widetilde{r}(y)(x-y)\widehat{T}$ indeed yields the symmetric expression $x^{2}-xy+y^{2}+bx+by+b^{2}$.

If $\widetilde{r}$ is a constant then $\widetilde{r}\neq0$ (since $-x$ is not symmetric), and the first condition expresses the constant multiple $\widetilde{r}\widehat{T}(x,y)$ of $\widehat{T}(x,y)$ as $x+U(x,y)$ for some symmetric polynomial $U$. Then $\widetilde{r}^{2}\widehat{T}(x,y)^{2}-\widetilde{r}\widehat{T}(x,y)(x-y)$ equals the product $(x+U)(y+U)$, which is symmetric since $U$ is. This completes the proof of the lemma.
\end{proof}

Another auxiliary result that we shall require is the following one.
\begin{lem}
Let $T$ and $Q_{0}$ be two non-zero polynomials in the variables $x$ and $y$ that are both divisible by $x-y$, such that the products $T \cdot sT$ and $Q_{0} \cdot sQ_{0}$ coincide. Then there exist polynomials $\varphi$ and $\psi$ such that $T(x,y)=\varphi(x,y)\psi(x,y)$ and $Q_{0}(x,y)=\varphi(x,y)\psi(y,x)$. \label{eqsymprod}
\end{lem}

\begin{proof}
We take $\tilde{\varphi}$ to be the greatest common divisor of $T$ and $Q_{0}$ (the choice of the precise scalar multiple will not play a role here) and let $\tilde{\psi}$ be the polynomial such that $T=\tilde{\varphi}\tilde{\psi}$. But then $Q_{0} \cdot sQ_{0}=T \cdot sT$ expands as $\tilde{\varphi}\cdot s\tilde{\varphi}\cdot\tilde{\psi} \cdot s\tilde{\psi}$, and since $\tilde{\varphi}$ divides $Q_{0}$ and thus $s\tilde{\varphi}$ divides $sQ_{0}$, we get that $\tilde{\psi} \cdot s\tilde{\psi}$ equals $(Q_{0}/\tilde{\varphi}) \cdot s(Q_{0}/\tilde{\varphi})$. But the $\gcd$ property of $\tilde{\varphi}$ implies that $\tilde{\psi}=T/\tilde{\varphi}$ is co-prime to $Q/\tilde{\varphi}$, so it must divide $s(Q_{0}/\tilde{\varphi})$, and therefore $s\tilde{\psi}$ divides $Q_{0}/\tilde{\varphi}$.

We can thus write $Q_{0}$ as $\varepsilon\tilde{\varphi} \cdot s\tilde{\psi}$, so that $sQ=s\tilde{\varepsilon} \cdot s\tilde{\varphi}\cdot\tilde{\psi}$, and by comparing the product of these expansions with the one for $Q_{0} \cdot sQ_{0}$, we deduce that $\varepsilon \cdot s\varepsilon=1$. This implies that $\varepsilon$ is a scalar satisfying $\varepsilon^{2}=1$, so that $\varepsilon\in\{\pm1\}$.

If $\varepsilon=+1$ then we are done by taking $\varphi=\tilde{\varphi}$ and $\psi=\tilde{\psi}$. Otherwise $\varepsilon=-1$ we recall that $x-y$ divides both $T$ and $Q_{0}$, so that it divides $\tilde{\varphi}$. Then $\tilde{\varphi}(x,y)=(x-y)\varphi(x,y)$, and by setting $\psi(x,y)=(x-y)\tilde{\psi}(x,y)$, we still get $T=\varphi\psi$. By observing that $\psi(y,x)=-(x-y)\tilde{\psi}(y,x)$, we deduce that $\varphi \cdot s\psi=-\tilde{\varphi} \cdot s\tilde{\psi}=Q_{0}$, as desired. This proves the lemma.
\end{proof}

\medskip

We obtain at the following result, for which we define $\zeta$ to be a scalar satisfying $\zeta^{2}-\zeta+1=0$ (namely a primitive 6th root of unity, away from characteristics 2 and 3). Note that the other root of that polynomial, which we denote by $\overline{\zeta}$, is given by either $1/\zeta$ of $1-\zeta$.
\begin{prop}
Take again $\pi$ and $\varpi$ as above, with the cubic braid relation, and assume that $\widetilde{Q}_{0}=0$. Then there are three cases.
\begin{enumerate}[(1)]
\item $\pi f=\varpi f=r(y)f$ for a univariate polynomial $r\neq0$.
\item $\varpi f=a(y+b)f$ and $\pi f=\begin{cases} a(x+b)[(\zeta x+\overline{\zeta}y+b)\partial f+\overline{\zeta}f] & \\ a(x+b)[(\overline{\zeta}x+\zeta y+b)\partial f+\zeta f] & \\ a(y\!+\!b)(\zeta x+\overline{\zeta}y+b)\partial f\!+\!a\big(x\!+\!\overline{\zeta}y\!+\!(1\!+\!\overline{\zeta})b\big)f & \\ a(y\!+\!b)(\overline{\zeta}x+\zeta y+b)\partial f\!+\!a\big(x\!+\!\zeta y\!+\!(1\!+\!\zeta)b\big)\!f\!, & \end{cases}$ for some constants $a\neq0$ and $b$, if $\zeta$ exists in our field of definition.
\item $\varpi$ is a multiple $\mu\operatorname{Id}$ of the identity operator and there exist two polynomials $\varphi$ and $\psi$ such that $\partial(\varphi\psi)=\mu$ and $\pi f=\varphi\partial(\psi f)$.
\end{enumerate} \label{degenQ}
\end{prop}
We remark that in characteristic 3 we have $\zeta=\overline{\zeta}=1/\zeta=1-\zeta=-1$, and the four cases for $\pi f$ in the second case reduce to the two expressions $a(x+b)[(b-x-y)\partial f-f]$ and $a(y+b)(b-x-y)\partial f+a(x-y)f$.

\begin{proof}
Lemma \ref{Q0univ} expresses $\widetilde{T}_{yz}$ as $(y-z)\widetilde{R}^{0}_{yz}$ and $\widetilde{R}^{0}_{yz}$ as $\widetilde{r}(y)$ for some univariate polynomial $\widetilde{r}$ (so that $\widetilde{R}^{0}_{xz}=\widetilde{r}(x)$), and Equation \eqref{beforeR=r} becomes
\begin{equation}
\widetilde{r}(y)^{2}T_{xy}(x-y)-\widetilde{r}(y)T_{xy}^{2}+\widetilde{r}(x)Q^{0}_{xy}Q^{0}_{yx}=0. \label{eqfromf}
\end{equation}

We now consider the coefficients of $sf$ in Lemma \ref{expansions}, which with our value of $\widetilde{T}$ and $\widetilde{R}_{0}$ produce, after canceling $\frac{(x-z)(y-z)}{x-y}$, the equality
\begin{equation}
Q^{0}_{xy}[\widetilde{r}(y)T_{xy}-\widetilde{r}(x)T_{yx}-\widetilde{r}(y)\widetilde{r}(x)(x-y)]=0. \label{fromsf}
\end{equation}
Now, if $Q_{0}=0$ as well then we can also replace $T_{xy}$ by $(x-y)r(y)$, and the last term in Equation \eqref{eqfromf} vanishes. Then we can cancel the multiplier $(x-y)^{2}$, as well as the non-vanishing product $r(y)\widetilde{r}(y)$, and this equality reduces to $\widetilde{r}(y)=r(y)$. This produces the first case.

We can thus henceforth assume that $Q_{0}\neq0$, and so after canceling $Q^{0}_{xy}$ from Equation \eqref{fromsf}, we deduce that $\widetilde{r}(x)$ divides $T_{xy}$. We also note that $\widetilde{r}(y)$ divides the rightmost term in Equation \eqref{eqfromf}, so that it divides the symmetric polynomial $Q^{0}_{xy}Q^{0}_{yx}$ (indeed, this is clear when it is constant, and otherwise none of its irreducible divisors divide $\widetilde{r}(x)$), so by symmetry $\widetilde{r}(x)$ also divides it. As $\widetilde{r}(x)^{2}$ then divides the last two terms in that equation, it must divide $T_{xy}$ as well (since when it is not a constant, its irreducible divisors will divide neither $\widetilde{r}(y)$ nor $x-y$). We thus write $T_{xy}$ as $\widetilde{r}(x)^{2}\widehat{T}_{xy}$, and after substituting it into Equation \eqref{fromsf} and canceling $Q^{0}_{xy}\widetilde{r}(x)\widetilde{r}(y)$, we obtain that $\partial[\widetilde{r}(x)\widehat{T}_{xy}]=1$. But this implies, via Lemma \ref{symimdi} and the $\partial$-positive pre-image $x$ of 1 under $\partial$, that $\widetilde{r}(x)\widehat{T}_{xy}-x$ is symmetric. In addition, substituting this expression for $T_{xy}$ into Equation \eqref{eqfromf} and dividing by $\widetilde{r}(x)^{2}\widetilde{r}(y)$ shows that $\widetilde{r}(x)^{2}\widehat{T}_{xy}^{2}-\widetilde{r}(y)\widehat{T}_{xy}(x-y)$ is the symmetric function $Q^{0}_{xy}Q^{0}_{yx}/\widetilde{r}(x)\widetilde{r}(y)$.

But now we are in the setting considered in Lemma \ref{fromsym}, meaning that there are two possible situations. In the first one we have $\widetilde{r}(x)=a(x+b)$ and $\widehat{T}=\frac{1}{a}$ for constants $a\neq0$ and $b$, and $\widehat{T}^{2}\widetilde{r}(x)^{2}-\widehat{T}\widetilde{r}(y)(x-y)$ was seen in the proof of that lemma to be $x^{2}-xy+y^{2}+bx+by+b^{2}$. When $\zeta$ exists in our field of definition, this polynomial decomposes as the product of $\zeta x+\overline{\zeta}y+b$ and $\overline{\zeta}x+\zeta y+b$. Comparing this expression with the value $Q^{0}_{xy}Q^{0}_{yx}/a^{2}(x+b)(y+b)$ of $Q^{0}_{xy}Q^{0}_{yx}/\widetilde{r}(x)\widetilde{r}(y)$ and recalling that $Q^{0}_{yx}$ is the $s$-image of $Q^{0}_{xy}$, we deduce that $Q^{0}_{xy}$ must be one of the four polynomials $a(x+b)(\zeta x+\overline{\zeta}y+b)$, $a(x+b)(\overline{\zeta}x+\zeta y+b)$, $a(y+b)(\zeta x+\overline{\zeta}y+b)$, and $a(y+b)(\overline{\zeta}x+\zeta y+b)$, or their inverses Recalling that Equation \eqref{Tpol} (with $P=0$) expresses $R_{0}(x,y)$ as $\frac{T(x,y)-Q_{0}(x,y)}{x-y}$, and $T(x,y)=\widehat{T}\widetilde{r}(x)^{2}=a(x+b)^{2}$, we obtain the asserted four options for $\pi f$ in the second case with these polynomials, and the inverses do not give more options as with them the difference $T-Q_{0}$ is not divisible by $x-y$.

In the second one $\widetilde{r}$ is a non-zero constant, which we denote by $\mu$ (so that $\varpi=\mu\operatorname{Id}$), and the symmetry of $e\widehat{T}_{xy}-x=T_{xy}/\mu-x$ is equivalent to the equality $\partial T_{xy}=\mu$. But then Equation \eqref{eqfromf} compares $Q^{0}_{xy}Q^{0}_{yx}$ with $T_{xy}\big(T_{xy}-\mu(x-y)\big)$, and the value of $\partial T_{xy}$ implies that the latter multiplier is just $T_{yx}$. We can then invoke Lemma \ref{eqsymprod}, and write $T=\varphi\psi$ and $Q_{0}=\varphi \cdot s\psi$ for polynomials $\varphi$ and $\psi$. Recalling the formula $\frac{T_{xy}f-Q^{0}_{xy}sf}{x-y}$ from the proof of Corollary \ref{Tcan} for $\pi f$, substituting these expressions for $T$ and $Q_{0}$ and recalling the definition of $\partial$ implies that this is the situation described in the third case. This completes the proof of the proposition.
\end{proof}

\begin{rmk}
Assuming, analogously to Proposition \ref{degenQ}, that $Q_{0}=0$ now, we get the same Case (1) there. When comparing the coefficients of $f$ and $\sigma f$, we obtain that if $\widetilde{Q}_{0}\neq0$ then $\widetilde{T}_{yz}=r(z)^{2}\widehat{T}(y,z)$ with $\delta[r(z)\widehat{T}_{yz}]=1$. Lemma \ref{fromsym} applies, with a very similar proof, for a univariate polynomial $r$, using the symmetry of $r(z)\widehat{T}(y,z)+z$ and $r(z)^{2}\widehat{T}(y,z)^{2}-r(y)\widehat{T}(y,z)(y-z)$ in $y$ and $z$, but now we take $r(z)=-a(z+b)$ to get $\widehat{T}=+\frac{1}{a}$ (with the same expression $y^{2}-yz+z^{2}+by+bz+b^{2}$) and when $r$ is a constant the symmetric product is $(U-y)(U-z)$. In Case (2) we then have \[\pi f=-a(y+b)f\mathrm{\ and\ }\varpi f=\begin{cases} a(z+b)[(\zeta y+\overline{\zeta}z+b)\partial f-\zeta f] & \\ a(z+b)[(\overline{\zeta}y+\zeta z+b)\partial f-\overline{\zeta}f] & \\ a(z+b)(\zeta y+\overline{\zeta}z+b)\partial f-a\big(\zeta y\!+\!z+(1\!+\!\zeta)b\big)f & \\ a(z+b)(\overline{\zeta}y+\zeta z+b)\partial f-a\big(\overline{\zeta}y\!+\!z+(1\!+\!\overline{\zeta})b\big)f & \end{cases}\] (note the signs), which in characteristic 3 reduce to $a(z+b)[(b-y-z)\partial f+f]$ and $a(z+b)(b-y-z)\partial f+a(y-z)f$. If $r$ is the constant $\mu$, then $\delta\widetilde{T}_{yz}=\mu$, the product $\widetilde{Q}^{0}_{yz}\widetilde{Q}^{0}_{zy}$ equals $\widetilde{T}_{yz}\widetilde{T}_{zy}$, and Case (3) becomes with $\pi=\mu\operatorname{Id}$ and $\varpi f=\widetilde{\varphi}\partial(\widetilde{\psi}f)$ for polynomials $\widetilde{\varphi}$ and $\widetilde{\psi}$ in $y$ and $z$ that satisfy $\delta(\widetilde{\varphi}\widetilde{\psi})=\mu$. \label{Qdegen}
\end{rmk}

Proposition \ref{degenQ} and Remark \ref{Qdegen} consider the case of two operators, involving three variables. In general, however, we will be interested in families of polynomial divided difference operators involving more variables and more operators. We then obtain the following consequence.
\begin{cor}
Consider, for $n\geq4$, a family $\{\pi_{i}\}_{i=1}^{n-1}$ of non-zero polynomial divided difference operators, again with the cubic braid relations. Then any operator $\pi_{i}$ for which the associated $Q_{0}$ polynomial vanishes a non-zero scalar multiple of $\operatorname{Id}$. Moreover, each of the operators $\pi_{i-1}$ and $\pi_{i+1}$ is either of the form $f\mapsto\varphi\partial(\psi f)$ where $\varphi$ and $\psi$ are polynomials whose product satisfies $\partial(\varphi\psi)=\mu$, or also equals $\mu\operatorname{Id}$. \label{ngeq4sc}
\end{cor}

\begin{proof}
We first consider $2 \leq i \leq n-2$ (which is possible since $n\geq4$). Since $i\geq2$, we can apply Proposition \ref{degenQ} with $\pi=\pi_{i-1}$ and $\varpi=\pi_{i}$. Then we have, in all three cases, that $\pi_{i}f=\widetilde{r}(x_{i+1})f$ for some non-zero univariate polynomial $\widetilde{r}$ (this was seen already in Lemma \ref{Q0univ}). But with $i \leq n-2$ we can also consider Remark \ref{Qdegen} with $\pi=\pi_{i}$ and $\varpi=\pi_{i+1}$ (or Lemma \ref{Q0univ} again), and obtain that $\pi_{i}f=r(x_{i-1})f$ for a non-zero univariate polynomial $\widetilde{r}$. But these two can happen together if and only if both polynomials are the same scalar $\mu$.

Assume now that $i=n-1$, and we take $\varpi=\pi_{n-1}$ and $\pi=\pi_{n-2}$ in Proposition \ref{degenQ}. Hence $\varpi f=r(x_{n-1})f$, and we need to show that $r$ is a scalar, which is automatic in Case (3). In Case (1) we have the same polynomial also for $\pi=\pi_{n-2}$, with $n-2\geq2$, and we already saw that this polynomial must then be a scalar. We need to show that Case (2) cannot happen, for which we observe, via the proof of Proposition \ref{degenQ}, that the polynomial $T$ associated with $\pi$ is $a(x_{n-2}+b)^{2}$ with $a\neq0$, and the corresponding polynomial $Q_{0}$ is also non-zero. When we consider the braid relation between $\pi_{n-3}$ and $\pi_{n-2}$, note that if the polynomial $Q_{0}$ that is associated with $\pi_{n-3}$ is non-zero, then Lemma \ref{Qneq0} shows that the two operators must produce the same polynomial $T$, which was seen to be non-zero. But then Proposition \ref{sfsigmaf} forces $T$ the quadratic form of $T(x_{n-2},x_{n-1})$ to be a multiple of $x_{n-2}x_{n-1}$, a property that our value of $T$ does not have.

It remains to consider the situation in Case (2) where $\pi_{n-3}$ has a vanishing polynomial $Q_{0}$, so that we are in the situation considered in Remark \ref{Qdegen} for $\pi=\pi_{n-3}$ and $\varpi=\pi_{n-2}$. But our situation is not the one described in Case (1) there since the polynomial $Q_{0}$ associated with $\pi_{n-2}$ is non-zero. Case (2) there would have implies that the polynomial $T$ of $\pi_{n-2}$ is of the form $c(x_{n-1}+d)^{2}$, which is impossible since our $T$ is $a(x_{n-2}+b)^{2}$. Finally, we saw in the proof of Proposition \ref{sfsigmaf} and Remark \ref{Qdegen} that in Case (3) the product $\varphi\psi$ is the polynomial $T$. But the condition that $\partial_{n-2}(\varphi\psi)$ is a constant does not hold for our $T$, which means that this last option for Case (2) for $\varpi=\pi_{n-1}$ and $\pi=\pi_{n-2}$ in Proposition \ref{degenQ} is impossible as well, and $\pi_{n-1}$ is as desired.

The remaining index $i=1$ is dealt with in an analogous manner, by invoking Remark \ref{Qdegen} for $\pi=\pi_{1}$ and $\varpi=\pi_{2}$, obtaining the result immediately in Case (3) and easily in Case (1), and excluding Case (2) by a similar consideration using the braid relations between $\pi_{2}$ and $\pi_{3}$ (with Lemma \ref{Qneq0} and Proposition \ref{sfsigmaf} in case the polynomial $Q_{0}$ corresponding to the latter does not vanish, and the cases in Proposition \ref{degenQ} when it does vanish). This proves the corollary.
\end{proof}

\section{Hecke Algebras and Commuting Operators \label{HecandCom}}

Recall that given constants $\mu$ and $\nu$, the Hecke algebra $\mathcal{H}_{\mu,\nu}$ (of degree $n$) is generated by $\{u_{i}\}_{i=1}^{n-1}$ satisfying the braid relations as well as the equality $u_{i}^{2}=\mu u_{i}+\nu$ for every $1 \leq i \leq n-1$. We wish to show that the polynomial divided difference operators $\{\pi_{i}\}_{i=1}^{n-1}$ from Theorem \ref{main} always generate such an algebra, and see which of the families involving degenerate polynomial divided difference operators also produce a algebra of that sort. In fact, when $n\geq4$ this will be the case for any family, except the one from Proposition \ref{degenQ} in case the polynomials there are not scalars. We will also determine when two families of polynomial divided difference operators that satisfy the braid relations generate commuting algebras.

For this we once again fix the index $i$, denote the variables $x_{i}$ and $x_{i+1}$ by $x$ and $y$ respectively, and write $s$, $\partial$, and $\pi$ for $s_{i}=s_{xy}$, $\partial_{i}=\partial_{xy}$, and $\pi_{i}$ respectively. We work with the first normalization in Proposition \ref{canform}, in which $\pi g=Q_{0}(x,y)\partial g+R_{0}(x,y)g$ for every function $g$ of $x$ and $y$, and we shorten the coefficients to $Q^{0}_{xy}$ and $R^{0}_{xy}$ as above. In addition, let another family $\{\widehat{\pi}_{i}\}_{i=1}^{n-1}$ of polynomial divided difference operators be given, whose $i$th element $\widehat{\pi}:=\widehat{\pi}_{i}$ takes a function $g=g(x,y)$ to $\widehat{Q}^{0}_{xy}\partial g+\widehat{R}^{0}_{xy}g$.

We now carry out the following evaluation.
\begin{lem}
For $\pi$ and $\widehat{\pi}$ as above and a function $f$ of $x$ and $y$ we have \[\pi(\widehat{\pi}f)=\big(Q^{0}_{xy}\partial\widehat{Q}^{0}_{xy}+R^{0}_{xy}\widehat{Q}^{0}_{xy}+Q^{0}_{xy}\widehat{R}^{0}_{yx}\big)\partial f+\big(Q^{0}_{xy}\partial\widehat{R}^{0}_{xy}+R^{0}_{xy}\widehat{R}^{0}_{xy}\big)f.\] \label{prodopers}
\end{lem}

\begin{proof}
Using the definition of $\pi$ and $\widehat{\pi}$ we obtain that \[\pi(\widehat{\pi}f)=Q^{0}_{xy}\partial\big(\widehat{Q}^{0}_{xy}\partial f+\widehat{R}^{0}_{xy}f\big)+R^{0}_{xy}\big(\widehat{Q}^{0}_{xy}\partial f+\widehat{R}^{0}_{xy}f\big).\] We apply Lemma \ref{Leibniz} to the first two terms while recalling that $\partial f$ is symmetric, and when we gather the resulting multipliers of $f$ and of $\partial f$, we obtain the asserted expression. This proves the lemma.
\end{proof}
In particular, the composition of two polynomial divided difference operators acting on the same variables is also a polynomial divided difference operator, whose form in the first normalization from Proposition \ref{canform} is immediately read off of Lemma \ref{prodopers}.

For Hecke algebras we deduce from Lemma \ref{prodopers} the following consequence.
\begin{cor}
A polynomial divided difference operator $\pi$, with polynomials $Q_{0}\neq0$, $R_{0}$, and $T$, satisfies the Hecke condition $\pi^{2}=\mu\pi+\nu\operatorname{Id}$ if and only if the equalities $\partial T_{xy}=\mu$ and $\partial(R^{0}_{xy}T_{yx})+R^{0}_{xy}R^{0}_{yx}=\nu$ hold. \label{forHecke}
\end{cor}

\begin{proof}
By taking $\widehat{\pi}=\pi$ in Lemma \ref{prodopers}, the multiplier of $\partial f$ in $\pi^{2}f$ becomes $Q^{0}_{xy}$ times $\partial Q^{0}_{xy}+R^{0}_{xy}+R^{0}_{yx}$, and as $T_{xy}=Q^{0}_{xy}+(x-y)R^{0}_{xy}$ by Equation \eqref{Tpol} (as in Remark \ref{degnondeg}), a simple calculation using Lemma \ref{Leibniz} shows that the latter combination is just $\partial T_{xy}$. Since for $\pi^{2}$ and $\mu\pi+\nu$ to coincide for every $f$ we need the coefficients of $\partial f$ and $f$ in them to be the same (as in, e.g., the last assertion in Lemma \ref{expansions}), the comparison of those of $\partial f$ and the assumption that $Q^{0}_{xy}\neq0$ yield the first equality.

From the coefficients of $f$ in $\pi^{2}f-\mu\pi f=\nu f$ we obtain that $\nu$ must equals the combination $Q^{0}_{xy}\partial R^{0}_{xy}+R^{0}_{xy}(R^{0}_{xy}-\mu)$. We express $Q^{0}_{xy}$ using $T_{xy}$ as above, recall that $(x-y)\partial R^{0}_{xy}=R^{0}_{xy}-R^{0}_{yx}$, cancel $(R^{0}_{xy})^{2}$, write $\mu$ as $-\partial T_{yx}$ by the previous equality, and apply Lemma \ref{Leibniz} to get that the expression that equals $\nu$ is the asserted one. Since our argument goes also in the opposite direction, the other implication is established as well. This proves the corollary.
\end{proof}

We can now establish the quadratic relation required for the Hecke property.
\begin{prop}
The operators $\{\pi_{i}\}_{i=1}^{n-1}$ from Theorem \ref{main} generate a Hecke algebra $\mathcal{H}_{b-c,\nu}$, where $\nu$ equals $e(e+c-b)$ in Case (1) and 0 in Case (2). \label{Hecke}
\end{prop}

\begin{proof}
We need to show that each operator $\pi=\pi_{i}$ from that theorem satisfies $\pi^{2}=(b-c)\pi+\nu\operatorname{Id}$ for the asserted value of $\nu$. The proof of Proposition \ref{sfsigmaf} shows that the polynomial $T$ satisfies $\partial T=b-c$ as Corollary \ref{forHecke} requires, and the proof of Proposition \ref{lastcoeff} establishes the second equality there, where Lemma \ref{restonQR} and Remark \ref{valofnu} establish the value of $\nu$ to be the desired one. This proves the proposition.
\end{proof}
One can write the Hecke equation for every $\pi=\pi_{i}$ in Proposition \ref{Hecke} as $(\pi-e\operatorname{Id})\big(\pi-(b-c-e)\operatorname{Id}\big)=0$. When $e$ equals neither 0 nor $b-c$, this expresses each such $\pi$ as an invertible operator, with $\pi^{-1}=\frac{\pi-(b-c)\operatorname{Id}}{e(e-b+c)}$.

It is clear that for $\mu\neq0$, the Hecke relation $\pi^{2}=\mu\pi$ (with $\nu=0$) is equivalent to $\pi$ and $\mu\operatorname{Id}$ satisfying the cubic braid relation. We can thus deduce from Corollary \ref{ngeq4sc} the general form of a family of polynomial divided difference operators that satisfy the braid relations and where at least one of the $Q_{0}$-polynomials vanish. Just like Proposition \ref{Hecke} shows that those from Theorem \ref{main} always produce Hecke algebras, we will get that when $n\geq4$, families with some vanishing $Q_{0}$'s also do the same.

Let $I$ be a subset of the integers between 1 and $n-1$, and take $i \in I$. We will say that $i$ is \emph{isolated} if neither $i-1$ nor $i+1$ lie in $I$, where for $i=1$ and for $i=n-1$, only the neighbor 2 or $n-2$ has to be considered. The elements of $I$ that are not isolated form a disjoint union of separated sets of consecutive integers, which we will call the \emph{intervals} of $I$. In particular, an interval of $I$ is not an isolated point, and is maximal in the sense that it cannot be extended to a larger set of consecutive integers that are all contained in $I$. With this terminology, we can now obtain the following result.
\begin{thm}
For $n\geq4$, let $\{\pi_{i}\}_{i=1}^{n-1}$ be a family of non-vanishing polynomial divided difference operators, in which we let $I$ be the set of indices $1 \leq i \leq n-1$ for which $\pi_{i}$ is not a scalar multiple of $\operatorname{Id}$. Then this family satisfies the braid relations and contains at least one vanishing $Q_{0}$-polynomial precisely when the following conditions hold:
\begin{enumerate}[(1)]
\item The complement of $I$ is not empty, and for every $i$ in that complement we have $\pi_{i}=\mu\operatorname{Id}$ for the same non-zero constant $\mu$;
\item For any isolated element of $I$ we have $\pi_{i}f=\varphi_{i}\partial_{i}(\psi_{i}f)$, where $\varphi_{i}$ and $\psi_{i}$ are functions of $x_{i}$ and $x_{i+1}$ that satisfy $\partial_{i}(\varphi_{i}\psi_{i})=\mu$;
\item The restriction of $\{\pi_{i}\}_{i=1}^{n-1}$ to an interval of $I$ yields a family like in Case (2) of Theorem \ref{main}, in which the parameters $a$, $b$, $c$, and $d$ also satisfy $b-c=\mu$.
\end{enumerate}
Moreover, every such family generates a Hecke algebra of the form $\mathcal{H}_{\mu,0}$. \label{withvanQ0}
\end{thm}

\begin{proof}
The quadratic braid relations are clear via Proposition \ref{quadcom}, and we focus on the cubic ones. It is clear that non-zero two multiples, say $\mu\operatorname{Id}$ and $\lambda\operatorname{Id}$, satisfy the cubic braid relation if and only if $\lambda=\mu$ (indeed, the scalar multiples $\lambda\mu^{2}$ and $\lambda^{2}\mu$ have to be the same). Case (3) in Proposition \ref{degenQ} and Remark \ref{Qdegen} shows that the operators from our Condition (2) satisfy the cubic braid relations with their neighbors from the complement of $I$, and Theorem \ref{main} implies the cubic braid relation among any consecutive operators that lie inside an interval of $I$. Moreover, as Condition (1) implies that no interval is the full set, each one must involve at least one operator at an end of an interval, and this operator has to satisfy the cubic braid relation with its neighbor $\mu\operatorname{Id}$. But Proposition \ref{Hecke} shows that it satisfies the Hecke condition $\pi_{i}^{2}=\mu\pi_{i}$, and we saw that this is equivalent to the desired braid relation. We have thus obtained that every such family satisfies the braid relations, and using the same equivalence we also obtain the generation of a Hecke algebra $\mathcal{H}_{\mu,0}$.

Conversely, Corollary \ref{ngeq4sc} shows that since $n\geq4$, all the operators that have a vanishing $Q_{0}$-polynomial are in the complement of $I$, so that this complement is not empty. We also saw that neighboring operators in the complement of $I$ are the same scalar multiple of $\operatorname{Id}$. This corollary produces Condition (2) as well, and the fact that we can extract the value of the scalar from $\varphi_{i}$ and $\psi_{i}$ implies that the scalar multiples of $\operatorname{Id}$ on both sides of an isolated element of $I$ are the same. Finally, no $T$-polynomial can vanish (by Proposition \ref{degenT}), so the form of the restriction of the family to every interval of $I$ is given by Theorem \ref{main}. But Proposition \ref{Hecke} then produces a Hecke condition for such a restriction, and we saw that this restriction has to satisfy a Hecke condition with $\nu=0$ for the cubic braid relation with a neighbor from the complement of $I$. Thus the restriction is as in Case (2) of that theorem, and again the value of the scalar multiples of $\operatorname{Id}$ on both sides are determined by the Hecke condition from that proposition, namely $b-c$ on both sides of the interval. Hence Conditions (3) and (1) are also established. This proves the theorem.
\end{proof}
Note that in Theorem \ref{withvanQ0}, the braid relation between an operator as in Case (2) of Theorem \ref{main} and $\mu\operatorname{Id}$ implies that the former operators have to be of the form $f\mapsto\varphi\partial(\psi f)$, or equivalently $f\mapsto\varphi\cdot(s\psi\partial f+\partial\psi \cdot f)$, for functions $\varphi$ and $\psi$. Indeed, the polynomial $T$ is either linear univariate (including a constant) if $a=0$, or the product of two strict linear polynomials, one in each variable, otherwise. Then the second line in Case (2) of Remark \ref{otherforms} is with $\varphi=T$ and $\psi=1$, the first one is obtained by taking $\varphi=1$ and $\psi=T$, and when $a\neq0$ (otherwise we covered all options, by Remark \ref{lesselts}) in the third line $\varphi$ is the polynomial in $x_{i+1}$ and $\psi$ is the one in $x_{i}$, while in the fourth one we interchange their roles.

We conclude the discussion about Hecke algebras by considering the remaining cases from Proposition \ref{degenQ} and Remark \ref{Qdegen} for $n=3$, as well as the families from Proposition \ref{degenT}. When the operator multiplies any function $f$ by a non-scalar polynomial, it does not satisfy any Hecke relation, as the only polynomials $r$ that satisfy $r^{2}=\mu r+\nu$ are scalars. It follows that for $n=3$ the families of non-vanishing polynomial divided difference operators, with non-vanishing $T$-polynomials, that satisfy the braid relation and the Hecke condition are those from Theorem \ref{main} and from the extension of Theorem \ref{withvanQ0} to $n=3$.

Similarly, the only Hecke relation that an operator $\pi$ taking $f$ to a polynomial times $s_{i}f$ is one of the form $\pi^{2}=\nu$ (the argument at the end of the proof of Lemma \ref{expansions} explains this). This again happens if and only if the polynomial is a scalar $\lambda$, and then $\nu=\lambda^{2}$. Hence the operators from Proposition \ref{degenT} satisfy the Hecke condition if and only if the polynomials $R_{i}$ there are scalars (and then they are the same scalar, by being almost equal). Note that the almost equality condition implies that for the operators from that proposition we have $\pi_{i}^{2}f=R_{i}(x_{i},x_{i+1})R_{i}(x_{i+1},x_{i})f$, and the multiplier is the value on $x_{i}$ and $x_{i+1}$ of the same symmetric polynomial for all $i$.

\medskip

For commutativity, we draw from Lemma \ref{prodopers} the following consequence.
\begin{cor}
The operators $\pi$ and $\widehat{\pi}$ from Lemma \ref{prodopers} commute precisely when the equalities $Q^{0}_{xy}\partial\widehat{Q}^{0}_{xy}=\widehat{Q}^{0}_{xy}\partial Q^{0}_{xy}$ and $Q^{0}_{xy}\partial\widehat{R}^{0}_{xy}=\widehat{Q}^{0}_{xy}\partial R^{0}_{xy}$ hold. \label{comindexi}
\end{cor}

\begin{proof}
We compare the expression for $\pi(\widehat{\pi}f)$ from that lemma with that of $\widehat{\pi}(\pi f)$ obtained by interchanging $Q^{0}_{xy}$ with $\widehat{Q}^{0}_{xy}$ and $R^{0}_{xy}$ with $\widehat{R}^{0}_{xy}$, and note that the term $R^{0}_{xy}\widehat{R}^{0}_{xy}f$ is the same on both sides. As in Lemma \ref{expansions}, the operators commute if and only if the coefficients of $f$ and $\partial f$ are the same in both compositions, and the second assert equality is that of the remaining coefficients of $f$ on the two sides. Now, when we consider the difference between the coefficients of $\partial f$, the parts involving $R_{0}$ and $\widehat{R}_{0}$ reduce to the difference between $Q^{0}_{xy}\partial\widehat{R}^{0}_{xy}$ and $\widehat{Q}^{0}_{xy}\partial R^{0}_{xy}$. Hence, given the second equality, the coefficients of $\partial f$ are the same if and only if the first equality holds as well. This proves the corollary.
\end{proof}

We can now establish the commutativity relations between two polynomial divided difference operators $\pi=\pi_{i}$ and $\widehat{\pi}=\widehat{\pi}_{i}$ with the same index $i$, when both operators are of the type considered in Theorem \ref{main} and Remark \ref{otherforms}. Since we now take one operator in each family, it will be convenient to gather Case (1) with the first two lines in Case (2) to obtain the \emph{extended Case (1)} (in which the value of $e$ is simply unrestricted). Moreover, Remark \ref{lesselts} shows that the remaining two lines in Case (2) are special cases of the extended Case (1) when $a=0$, so that we refer as the \emph{strict Case (2)} to the operators obtained by the last two lines in Case (2) under the assumption that $a\neq0$.
\begin{prop}
Let $\pi$ be as in Theorem \ref{main} and Remark \ref{otherforms}, and let $\widehat{\pi}$ be another such operator, that is based on the coefficients $\widehat{a}$, $\widehat{b}$, $\widehat{c}$, and $\widehat{d}$, and also $\widehat{e}$ in the extended Case (1). Assuming that $\pi$ and $\widehat{\pi}$ are not scalar multiples of one another, they commute if and only if their description as in Proposition \ref{canform} with $Q_{0}$, $R_{0}$, $\widehat{Q}_{0}$, and $\widehat{R}_{0}$ is in one of the following situations.
\begin{enumerate}[(1)]
\item $\pi$ and $\widehat{\pi}$ are operators from the extended Case (1) that are not scalar multiples of one another, and we have the equalities $c-b+2e=\widehat{c}-\widehat{b}+2\widehat{e}=0$.
\item There are constants $a$, $b$, $c$, $d$, and $e$ with $ad=bc$, $b \neq c$, and $b-c\neq2e$ such that $\pi f=\big(axy+(c+e)x+(b-e)x+d\big)\partial f+(b-c-e)f$ and $\widehat{\pi}f$ is a scalar multiple of $\big(axy+(c+e)x+(b-e)x+d\big)\partial f-ef$.
\item $\pi f=\big(ax^{2}+(b+c)x+d\big)\partial f-(ax+c)f$ for constants $a\neq0$, $b \neq c$, and $d$ with $ad=bc$, and $\widehat{\pi}f$ is a scalar multiple of $\big(ax^{2}+(b+c)x+d\big)\partial f-(ax+b)f$.
\item $\pi f=\big(ay^{2}+(b+c)y+d\big)\partial f+(ay+b)f$ for $a$, $b$, $c$, and $d$ as in Case (3), and $\widehat{\pi}f$ is a scalar multiple of $\big(ay^{2}+(b+c)y+d\big)\partial f+(ay+c)f$.
\end{enumerate} \label{sameindex}
\end{prop}
We have once again adopted the notation $x=x_{i}$ and $y=x_{i+1}$ as above.

\begin{proof}
It is clear from Remarks \ref{otherforms} and \ref{lesselts} and our definition that $\partial R_{0}$ vanishes when $\pi_{i}$ is in the extended Case (1) and equals a non-zero constant in the strict Case (2), and the same for $\partial\widehat{R}_{0}$ and $\widehat{\pi}$. Since $Q_{0}$ and $\widehat{Q}_{0}$ do not vanish in Theorem \ref{main}, the second equality in Corollary \ref{comindexi} implies that $\partial R_{0}\neq0$ if and only if $\partial\widehat{R}_{0}\neq0$, so that $\pi_{i}$ and $\widehat{\pi}$ are either both in the extended Case (1) or both in the strict Case (2). Consider first the latter case, and since we work up to scalar multiplication, we may assume that $\widehat{a}=a$, and then the equality in question implies that $\widehat{Q}_{0}=Q_{0}$, from which the first equality in Corollary \ref{comindexi} immediately follows. Thus both are of the same type (either a quadratic polynomial in $x_{i}$ or one in $x_{i+1}$), and we also have $\widehat{d}=d$ and $\widehat{b}+\widehat{c}=b+c$. But since the equalities $ad=bc$ and $\widehat{a}\widehat{d}=\widehat{b}\widehat{c}$ imply that $\widehat{b}\widehat{c}=bc$ as well, and we assume that $\pi$ and $\widehat{\pi}$ are not scalar multiples of one another, we must have $\widehat{b}=c$, $\widehat{c}=b$, and $b \neq c$. This produces the asserted Cases (3) and (4), according to the type of $Q_{0}$ and $\widehat{Q}_{0}$.

We thus now assume that $\pi_{i}$ and $\widehat{\pi}$ are both as in the extended Case (1), and we only have to consider the first equality in Corollary \ref{comindexi}. Once again both sides vanish simultaneously, with $Q_{0}$ and $\widehat{Q}_{0}$ non-zero, and we recall that $\partial Q_{0}$ and $\partial\widehat{Q}_{0}$ are now scalars. If the scalars are non-zero then $Q_{0}$ and $\widehat{Q}_{0}$ are again scalar multiples of one another, and without loss of generality we can once again assume that they are equal. In this case $\widehat{a}=a$ and $\widehat{d}=d$ once more, and we have $\widehat{c}+\widehat{e}=c+e$ and $\widehat{b}-\widehat{e}=b-e$, as well as $\widehat{b}\widehat{c}=bc$ as before. Since $\pi$ and $\widehat{\pi}$ are not scalar multiples of one another, this is only possible when again $\widehat{b}=c$ and $\widehat{c}=b$ with $b \neq c$, and $\widehat{e}=c-b+e$. This yields the operators described in the asserted Case (2).

The remaining situation is where $\pi_{i}$ and $\widehat{\pi}$ are in the extended Case (1) and the scalars $\partial Q_{0}$ and $\partial\widehat{Q}_{0}$ vanish. Then the operators commute by Corollary \ref{comindexi}. As the scalars in question are $c-b+2e$ and $\widehat{c}-\widehat{b}+2\widehat{e}=0$, this is the situation considered in the asserted Case (1). As all of our arguments are invertible, the operators in all cases indeed commute (and are not scalar multiples of one another). This completes the proof of the proposition.
\end{proof}
Note that in the three last cases in Proposition \ref{comindexi}, the operators from Theorem \ref{main} and Remark \ref{otherforms} that commute with $\pi$ are the multiples of $\pi$ and those of $\pi-(b-c)\operatorname{Id}$ (hence just those of $\pi$ when $b=c$), and the latter is a multiple of the inverse of $\pi$ when $\pi$ is invertible via Proposition \ref{Hecke}. The operators from the first case in Proposition \ref{comindexi}, which thus all commute with one another, are invertible in this sense when $b \neq c$ (since then $e$ equals neither 0 nor $b-c$), and not otherwise. The latter are also precisely the operators in the extended Case (1) whose $Q_{0}$-polynomial is symmetric (as was seen in the proof). Note that in characteristic 2, Case (1) of Proposition \ref{comindexi} occurs for every $e$ and every $\widehat{e}$ in the extended Case (1) when $b=c$ and $\widehat{b}=\widehat{c}$, and for none in case the latter equalities do not hold.

The operators associated with intervals in Theorem \ref{withvanQ0} are special cases of those from Theorem \ref{main} and Remark \ref{otherforms}, hence their commutation relation is already covered in Proposition \ref{comindexi}. When considering the more general operators than can show up for isolated indices, i.e., those from Case (3) of Proposition \ref{degenQ} and Remark \ref{Qdegen} (or Corollary \ref{ngeq4sc}), the determination of commutation of these operators with themselves or with those from Theorem \ref{main} and Remark \ref{otherforms} is more involved. We only remark that in the extended Case (1) of the latter, since its $R_{0}$-polynomial is annihilated by $\partial$, so does the product $\partial\varphi\partial\psi$, so that either $\varphi$ or $\psi$ is symmetric. But then it divides the non-zero scalar $\partial(\varphi\psi)=\mu$, so that the scalar $\partial Q_{0}=b-c-2e$ must be non-zero, and the operators are either the appropriate scalar multiples of $f \mapsto Q_{0}\partial f$, or of $f\mapsto\partial(sQ_{0} \cdot f)$ (with the scalar multiplier determined by $\mu$). These operators are the translations $\pi-(b-c-e)\operatorname{Id}$ and $\pi-e\operatorname{Id}$ of $\pi$. As for commuting with the operators from Proposition \ref{degenT}, with vanishing $T$-polynomials, they can only commute with the operators from the first case in Proposition \ref{comindexi} (by the symmetry of $Q_{0}$ and the vanishing of $\partial R_{0}$), and they indeed commute if and only if the associated polynomial $R_{i}$ from that proposition is symmetric as well.

We conclude by remarking that there are virtually no pairs of families, as in Theorems \ref{main} or \ref{withvanQ0}, where all the operators in one family commute with all the others in the second one. while Proposition \ref{comindexi} deals with the operators of the same index in two families, and Proposition \ref{quadcom} shows that commutativity is trivial when the indices differ by 2 or more, the condition that operators with consecutive indices commute is very restrictive. Indeed, comparing the expressions for $\pi\varpi f$ and $\varpi\pi f$ in the proof of Lemma \ref{expansions} shows, via the coefficients of $s\sigma f$ and $\sigma sf$, that the equality $\pi\varpi=\varpi\pi$ can hold only if either $Q^{0}_{xy}$ with $\widetilde{Q}^{0}_{yz}$ vanish. Then, via Corollary \ref{ngeq4sc}, we obtain the case where one family only involved multiples of $\operatorname{Id}$, which trivially commute with all operators.

\noindent\textsc{Einstein Institute of Mathematics, the Hebrew University of Jerusalem, Edmund Safra Campus, Jerusalem 91904, Israel}

\noindent E-mail address: zemels@math.huji.ac.il

\end{document}